\documentclass[ejs,preprint]{imsart}       
\usepackage{graphicx,tikz}
\usepackage
{color}

\usepackage{subfigure}

\usepackage{natbib}

\usepackage{algorithm}
\usepackage{algorithmic}

\usepackage{hyperref}

\usepackage{amsmath,amssymb,amsfonts}
\usepackage{dsfont}
\usepackage{stmaryrd} 
\usepackage{array} 
\usepackage{nonfloat,mathrsfs,booktabs}
\usepackage{numprint}

\usepackage{amsthm}

\def\diag{{\rm diag}}
\def\1{\mathds 1}

\def\hat{\widehat}
\def\Ex{{\mathbf E}}
\def\prob{{\mathbf P}}
\def\Var{{\mathbf V}}

\def\trace{\text{\rm trace}}

\def\bu{{\boldsymbol u}}
\def\bv{{\boldsymbol v}}

\def\bfone{{\mathbf 1}}
\def\bfc{{\mathbf c}}

\def\bfB{{\mathbf B}}
\def\bfM{{\mathbf M}}
\def\bfD{{\mathbf D}}

\def\bfY{{\mathbf Y}}
\def\bfX{{\mathbf X}}
\def\bfI{{\mathbf I}}
\def\bfS{{\mathbf S}}
\def\bfW{{\mathbf W}}
\def\bfA{{\mathbf A}}
\def\bfSigma{\mathbf\Sigma}
\def\bfXi{\mathbf\Xi}
\def\bfDelta{\mathbf\Delta}

\def\bfOmega{\mathbf\Omega}

\def\bbeta{\boldsymbol\beta}
\def\bmu{\boldsymbol\mu}
\def\bphi{\boldsymbol\phi}
\def\bxi{\boldsymbol\xi}

\newcommand{\field}[1]{\mathbb{#1}}

\newcommand{\R}{\field{R}}

\def\phiRV{\widehat{\phi}^{{\,\rm RV}}}
\def\bphiRV{\widehat{\bphi}{}^{{\,\rm RV}}}
\def\phiRML{\hat\phi^{\,\rm RML}}
\def\bphiRML{\hat\bphi{}^{\,\rm RML}}
\def\phiML{\hat\phi^{\,\rm SML}}
\def\bphiML{\hat\bphi{}^{\,\rm SML}}

\def\bphiPML{\hat\bphi{}^{\,\rm PML}}

\DeclareMathOperator*{\argmin}{arg\,min}

\theoremstyle{plain}
\newtheorem{prop}{Proposition}    

\arxiv{math.ST/1504.04696}
\begin{document}


\begin{frontmatter}
\title{On estimation of the diagonal elements of a sparse precision matrix
}
\runtitle{Estimation of the diagonal elements of a sparse matrix}

\begin{aug}
\author{Samuel Balmand \ead[label=e1]{samuel.balmand@ensg.eu}}
\and
\author{Arnak S. Dalalyan\ead[label=e2]{arnak.dalalyan@ensae.fr}}

\address{Universit\'e Paris-Est, IGN, MATIS, ENSG, F-77455, Marne-la-Vall\'ee, France \\ \printead{e1}}
\address{ENSAE ParisTech \& CREST, 3 av. P. Larousse, 92245 Malakoff, France\\ \printead{e2}}

\runauthor{Balmand and Dalalyan}

\end{aug}

\begin{abstract}
In this paper, we present several estimators of the diagonal elements of the inverse of the covariance matrix, called precision matrix,
of a sample of independent and identically distributed random vectors. The main focus is on the case of high dimensional vectors having a sparse
precision matrix. It is now well understood that when the underlying distribution is Gaussian, the columns of the precision matrix can be
estimated independently form one another by solving linear regression problems under sparsity constraints. This approach leads to a computationally
efficient strategy for estimating the precision matrix that starts by estimating the regression vectors, then estimates the diagonal entries
of the precision matrix and, in a final step, combines these estimators for getting estimators of the off-diagonal entries. While the step of
estimating the regression vector has been intensively studied over the past decade, the problem of deriving statistically accurate estimators
of the diagonal entries has received much less attention. The goal of the present paper is to fill this gap by presenting four estimators---that
seem the most natural ones---of the diagonal entries of the precision matrix and then performing a comprehensive empirical
evaluation of these estimators.  The estimators under consideration are the residual variance, the relaxed maximum likelihood, the symmetry-enforced
maximum likelihood and the penalized maximum likelihood. We show, both theoretically and empirically, that when the aforementioned regression vectors
are estimated without error, the symmetry-enforced maximum likelihood estimator has the smallest estimation error. However, in a more realistic
setting when the regression vector is estimated by a sparsity-favoring computationally efficient method, the qualities of the estimators become
relatively comparable with a slight advantage for the residual variance estimator.
\end{abstract}




\end{frontmatter}

\section{Introduction}
\label{sect:introduction}

We consider the problem of precision matrix estimation that has been extensively studied in recent years partly because of its tight relation
with the graphical models. More precisely, assuming that we observe $p$ features on $n$ individuals, an interesting object
to display is the graph of associations between the features, especially when the number of features is large. The associations may be of different
type: linear correlations, partial correlations, measures of independence and so on. A measure of association between the features, which is
particularly relevant for Gaussian \citep{Lauritzen} and, more generally, non paranormal distributions \citep{Liu_2009,Lafferty_2012} is the
partial correlation. This leads to a Gaussian graphical model in which two nodes are connected by an edge if the partial correlation between
the features corresponding to these two nodes is nonzero, which is equivalent to the nonzeroness of the corresponding entry of the precision matrix \citep[Proposition 5.2]{Lauritzen}. The graph constructed in such a way relies on the population precision matrix, which is not
available in practice. Therefore, an important statistical problem is to infer this graph from $n$ iid observations of the $p$-dimensional
feature-vector. In view of the aforementioned connection with the precision matrix, the estimated graph may be deduced from the estimated
precision matrix by comparing its entries with a suitably chosen threshold.

Another important problem for which the precision matrix estimation is relevant\footnote{In the case of linear
discriminant analysis for binary classification, a simpler approach consisting in replacing the sparsity of the precision
matrix by the sparsity of the product of the latter with the difference of the class means has been proposed and studied
by \cite{Cai_jasa11}.} is the linear \citep{Fisher36}
or quadratic discriminant analysis \citep{Anderson}. Indeed, the decision boundary in the binary or
multi-class  classification problem---under the assumption that the conditional distributions of the features
given the class are Gaussian---is defined in terms of the precision matrix. In order to infer this decision
boundary from data, it is therefore relevant to start with estimating the precision matrix. The simplest way
of estimating the latter is by inverting the sample covariance matrix or, if the inverse does not exist, by
computing the pseudo-inverse of the sample covariance matrix. However, when the dimension $p$ is such that the
number of unknown parameters $p(p+1)$ is comparable to or larger than the sample-size $n$, the (pseudo-)inversion
of the sample covariance matrix leads to very poor results. To circumvent this shortcoming, additional assumptions
on the precision matrix should be imposed which should preferably be realistic, interpretable and lead to
statistically and computationally efficient estimation procedures. The sparsity of the precision matrix offers
a convenient setting in which these criteria are met.


To present in a more concrete fashion the content of the present work, let $\bfX$ be a $n\times p$ random matrix
representing the values of $p$ variables observed on $n$ individuals. Assume that the rows of the matrix $\bfX$
are independent and Gaussian with mean $\bmu^*$ and covariance matrix $\bfSigma^*$. The inverse of $\bfSigma^*$,
called the precision matrix and  denoted by $\bfOmega^* = (\omega^*_{ij})$, is an object of central interest
since---as mentioned earlier---it encodes the conditional dependencies between pairs of variables given the values
of all the other variables. Based on the precision matrix, the graph $\mathscr G^*$ of relationships between the
$p$ variables is constructed as follows: each node of the graph represents a variable and two nodes $i$ and $j$
are connected by an edge if and only if $\omega^*_{ij}\not=0$. Estimating this graph from a sample of size $n$
represented by the rows of $\bfX$ is a challenging statistical problem that has attracted a lot of attention in
the past decade. In a frequently encountered situation of the dimension $p$ comparable to or even larger than $n$,
a commonly used assumption is the sparsity of the graph $\mathscr G^*$. Namely, it is assumed that the maximal
degree of the nodes of $\mathscr G^*$ is much smaller than $p$ (see, for instance, \cite{MeinshausenBuhl,Yuan2007}
for early references).

Most approaches of estimating sparse precision matrices that gained popularity in recent years rely on weighted $\ell_1$-penalization of the off-diagonal elements of the
precision matrix; recent contributions on the statistical aspects of this approach can be found in \cite{Yuan,CaiLiuLuo,SunZhang13,Cai2015}
and the references therein. The rationale behind this approach is that the weighted $\ell_1$-penalty can be viewed as a convexified version
of the $\ell_0$-penalty, the latter being understood as the number of nonzero elements. The convexity of the penalty in conjunction with the convexity of the data fidelity term leads to estimators that can be efficiently computed by convex programming
\citep{FriedmanHT,BanerjeeGAspr}.

To further improve the computational complexity, it is possible to split the problem of estimating $p^2$ entries of the precision matrix into $p$ independent problems of estimating the $p$-dimensional columns of it \citep{MeinshausenBuhl}. To this end, the matrix $\bfOmega^*$ is written as $\bfB^*\bfD^*$, where $\bfD^*$
is a diagonal matrix while $\bfB^*$ is a $p\times p$ matrix with all diagonal entries equal to one.
Each columns of the matrix $\bfB^*$ can be estimated by regressing one column of the data matrix $\bfX$ on all the remaining columns. In the context of high dimensionality and sparse precision matrix, this can be performed by sparsity favoring methods \citep{BuhlmanvdGeerBook} such as the Lasso \citep{tibshirani96}, the Dantzig selector \citep{CandesTao}, the square-root Lasso \citep{BelloniChernozhukovWang}, etc. A crucial observation at this stage is that the sparsity patterns, \textit{i.e.}, the locations of nonzero entries, of the matrices $\bfB^*$ and $\bfOmega^*$ coincide. In particular, the degree of the $j$-th node in the graph $\mathscr G^*$ is equal to the number of nonzero entries of the $j$-th column of $\bfB^*$, for every $j=1,\ldots,p$.

Once the columns of $\bfB^*$ successfully estimated, one needs to estimate the diagonal matrix $\bfD^*$, the diagonal entries of which coincide with those of the precision matrix $\bfOmega^*$. This step is necessary for recovering the precision matrix (both diagonal and nondiagonal entries) but it is also important for constructing the graph\footnote{We put an emphasize on this last point since we did not find it in the literature.} of conditional dependencies. Of course, the latter can be estimated by thresholding the entries of the estimator of $\bfB^*$ without resorting to an estimator of $\bfD^*$, but the choice of the threshold is in this case a difficult issue deprived of clear statistical interpretation. In contrast with this, if along with an estimator of $\bfB^*$, an estimator of $\bfD^*$ is available, then one may straightforwardly estimate the partial correlations and threshold them to infer the graph of conditional dependencies. In this case, the threshold has a more clear statistical
meaning since the partial correlations are in absolute value bounded by one.

\begin{figure*}
\centering
\includegraphics[width=\textwidth]{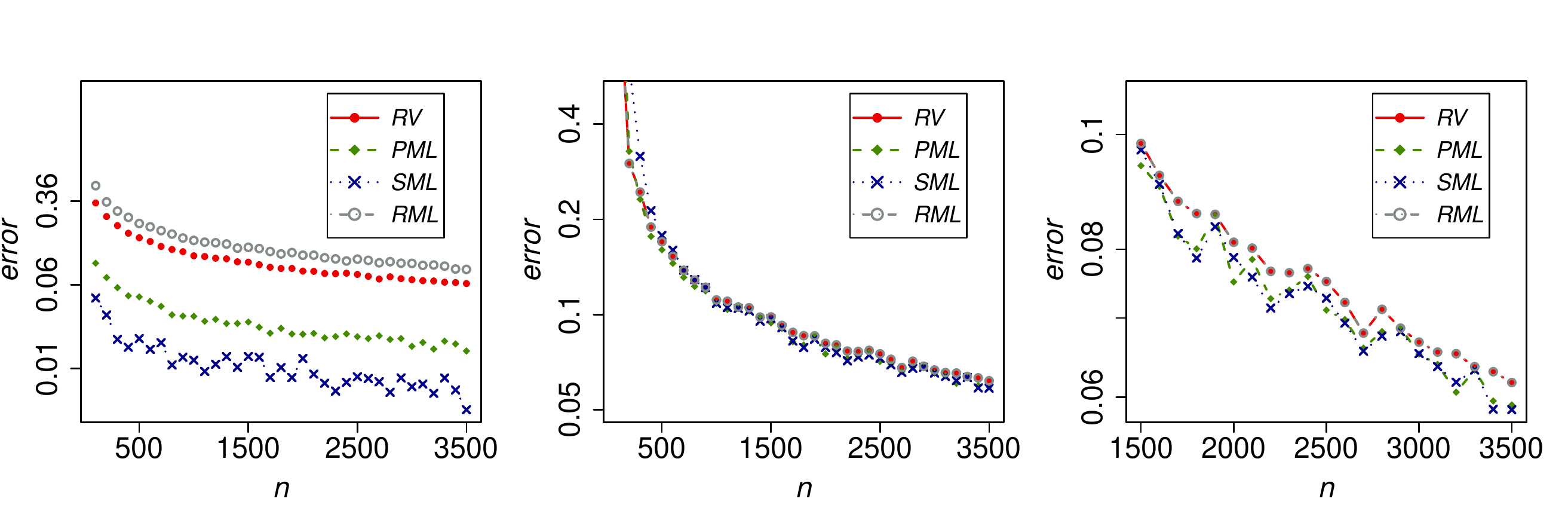}
\hbox to \hsize{\hfill\hfill\hfill $\hat\bfB$ coincides with $\bfB^*$ \hfill\hfill\hfill $\hat\bfB$ estimated by SqRL$+$OLS
\hfill\hfill $\hat\bfB$ estimated by SqRL$+$OLS: zoom}
\caption{The average $\ell_2$-error (computed from 50 independent trials)  of the four  estimators considered in this work as a function
of the sample size. The plots concern Model 2 described in Section~\ref{ssec:exp} and dimension $p=60$. One can observe, in particular, that when $\bfB^*$ is estimated without error (left panel), the estimators SML and PML improve on the residual variance and relaxed maximum likelihood estimators.}
\label{fig:globfig}
\end{figure*}

It follows from the above discussion that the problem of estimating the matrix $\bfD^*$ built from the diagonal entries of the precision matrix is an important ingredient of the estimation of the precision matrix and the graph of conditional dependencies between the features. The purpose of the present work is to propose several natural estimators of $\bfD^*$ and to study their statistical properties, essentially from an empirical point of view. Combining standard arguments, we present four estimators, termed residual variance (RV), relaxed maximum likelihood (RML), symmetry-enforced maximum likelihood (SML) and penalized maximum likelihood (PML). The first one, residual variance, is the most commonly used estimator when the matrix $\bfB^*$ is estimated column-wise by a sparse linear regression approach briefly mentioned in the foregoing discussion. The other three methods considered in this paper are based on the principle of likelihood maximization under various approaches for handling the prior
information. In order to give the reader a foretaste of the content of next sections, we present in Figure~\ref{fig:globfig} the accuracy of the four methods of estimating the diagonal elements of the precision matrix on a synthetic data-set. More details are given in Section~\ref{ssec:exp}.

\section{Notation and preliminaries on precision matrix estimation}

This section introduces notation used throughout the paper and presents some preliminary material on sparse precision matrix estimation.

\subsection{Notation}

For an unknown parameter $\theta$ we note $\theta^*$ its true value.
As usual, $\mathcal N_p(\bmu^*,\bfSigma^*)$ is the Gaussian distribution in $\R^p$ with mean $\bmu^*$ and covariance matrix $\bfSigma^*$.
The expectation of a random vector $\bfX$ is denoted by $\Ex(\bfX)$ and its covariance matrix by $\Var(\bfX)$.
We denote by $\bfone_n$ the vector from $\R^n$ with all the entries equal to $1$ and by
$\bfI_n$ the $n\times n$ identity matrix.  We write $\mathds{1}$ for the indicator function,
which is equal to 1 if the considered condition is satisfied and 0 otherwise.
The smallest integer greater than or equal to $x \in \R$ is denoted by $\lceil x \rceil$.
In what follows, $[p] :=\{1,\ldots,p\}$ is the set of positive integers from $1$ to $p$.
For $i \in [p]$, the complement of the singleton $\{i\}$ in $[p]$ is denoted by $i^c$. For a vector $\bv\in\R^p$, $\bfD_\bv$ stands for the $p\times p$ diagonal matrix satisfying $(\bfD_\bv)_j = \bv_j$
for every \(j\in[p]\).
The matrix build keeping only the diagonal of a square matrix \(\bfM\) is denoted by \(\diag(\bfM)\).

The transpose of the matrix $\bfM$ is denoted by $\bfM^\top$. If this matrix is square, we note $|\bfM|$ its determinant.
For a $n\times p$ matrix $\bfM$, the vector of the elements of the $k$th row (resp.\ the $j$th column) whose indices are given
by the subset $J$ of $[p]$ (resp.\ $K$ of $[n]$) is denoted by $\bfM_{k,J}$ (resp.\ $\bfM_{K,j}$). In particular,
the vector made of all the elements of the $j$th column of the matrix $\bfM$ at the exception of the element of the
$k$th row is given by $\bfM_{k^c,j}$. Moreover, the whole $k$th row (resp.\ $j$th column) of $\bfM$ is denoted by
$\bfM_{k,\bullet}$ (resp.\ $\bfM_{\bullet,j}$).
We use the following notation for the (pseudo-)norms of matrices: if $q_1,q_2>0$, then
\begin{equation*}
{\|\bfM\|}_{q_1,q_2} =
\left\{\sum_{i=1}^n \|\bfM_{i,\bullet}\|_{q_1}^{q_2}\right\}^{1/q_2} .
\end{equation*}
With this notation, $\|\bfM\|_{2,2}$ and $\|\bfM\|_{1,1}$ are the Frobenius and the element-wise $\ell_1$-norm of $\bfM$,
respectively. The sample covariance matrix of the data points $\{\bfX_{k,\bullet}\}_{k\in[n]}$ is defined by
\begin{equation*}
\bfS_n= \frac1{n} \sum_{k=1}^n (\bfX_{k,\bullet}^\top - \hat\bmu) (\bfX_{k,\bullet}^\top - \hat\bmu)^\top = \frac1{n} (\bfX - \bfone_n \hat\bmu^\top)^\top (\bfX - \bfone_n \hat\bmu^\top) ,
\end{equation*}
where $\hat\bmu$ is either the sample mean  $\frac1{n} (\bfone_n^\top \bfX)^\top$ (when the mean $\bmu^*$ is unknown) or
the theoretical mean $\bmu^*$ (when it is considered as known).

\subsection{Preliminaries}

Throughout the paper we will present estimators of the diagonal elements of the precision matrix in the case of a general multidimensional Gaussian distribution, but in all theoretical developments we will assume that the marginals of $\bfX$ are standard Gaussian distributions,
\textit{i.e.}, $\bmu^*=0$ and $\bfSigma_{jj}^*=1$ for every $j\in[p]$. This assumption is reasonable, since we are concerned with
the problems in which the sample size is large enough to consistently estimate the individual  means and the individual variances
of the variables. So, one can always center the variables by the sample mean and divide by the sample standard deviation to get
close to the assumption\footnote{Unless expressly
stated otherwise, in the whole article, $1 \leq i,j \leq p$ and $1 \leq k \leq n$.} that random variables $\bfX_{1,j},\ldots,\bfX_{n,j}$
are i.i.d.\ $\mathcal N(0,1)$ for every $j$.

Let us recall that the precision matrix is closely related to the problem of regression of one feature on all the others.
Indeed, there exists a $p\times p$ matrix $\bfB^*$ and two vectors $\bfc^*,\bphi^*\in\R^p$ such that
\begin{equation}
 \bfX_{\bullet,j} = c_j^* \bfone_n - \bfX_{\bullet,j^c} \bfB^*_{j^c,j} + \phi_j^*{}^{1/2}\, \bxi_j, \label{eq:model}
\end{equation}
where $\bxi_j$ is drawn from $\mathcal{N}_n(0, \bfI_n)$ and is independent of $\bfX_{\bullet,j^c}$.
According to the theorem on normal correlations \citep{Marsaglia}, the regression coefficients
$\bfB^*_{j^c,j} \in \R^{p-1}$ and the variance $\phi^*_j \in  \R$ of residuals
can be expressed in terms of the elements of the precision matrix $\bfOmega^*$ as follows:
\begin{equation}
 \bfB^*_{ij} = {\omega^*_{ij}}/{\omega^*_{jj}} ,  \label{eq:normalCorr} \qquad
 \phi^*_j = 1/{\omega^*_{jj}}, \tag{NC}
\end{equation}
whereas $c_j^* = \mu_j^*+(\bmu_{j^c}^*)^\top\bfB^*_{j^c,j}=(\bmu^*)^\top\bfB^*_{\bullet,j}$.
If we assume that $\bmu^* = 0$ then $c^*_j = 0$ for any $j$. With these notation, the precision
matrix can be written as $\bfOmega^* = \bfB^*\bfD_{\bphi^*}^{-1}$.

Several state-of-the-art methods for estimating sparse precision matrices proceed in two steps \citep{MeinshausenBuhl,
CaiLiuLuo, TIGER, SunZhang13}. The first step consists in estimating the matrix $\bfB^*$ and the vector $\bphi^*$ by solving
the sparse linear regression problems (\ref{eq:model}) for each $j$, while in the second step an estimator of the
matrix $\bfOmega^*$ is inferred from the estimators of $\bfB^*$ and $\bphi^*$ using relations (\ref{eq:normalCorr}).
The goal of the present work is to explore both theoretically and empirically different possible strategies for this
second step.

The square-root Lasso is perhaps the method of estimating the matrix $\bfB^*$ that offers the
best trade-off between the computational and the statistical complexities. It can be redefined as follows: scaled Lasso
estimates the matrix $\bfB^*$ by solving the convex optimization problem
\begin{equation}\label{step:1}
\hat\bfB = \argmin_{\substack{\bfB\in \R^{p\times p}\\ \bfB_{jj}=1}}\limits\min_{\bfc\in\R^p}
\Big\{{\|\bfX\bfB - \bfone_n\bfc^\top\|}_{2,1}+\lambda{\|\bfB\|}_{1,1}\Big\},
\end{equation}
where the first $\min$ is over all matrices $\bfB$ having all their diagonal entries equal to $1$.
The tuning parameter $\lambda > 0$ corresponds to the penalty level.
The purpose of the penalization is indeed to get a precision matrix estimate which fits the sparsity assumption.
As the penalty of a matrix $\bfB$ is its $\|\cdot\|_{1,1}$ norm, the resulting precision matrix estimator is expected
to be sparse in the sense that its overall number of non-zero elements should be small. In addition, one can check that
computing a solution to problem (\ref{step:1}) is equivalent to computing each column of $\hat\bfB$ separately (and
independently)  by solving the optimization problem
\begin{equation}\label{step:1a}
\hat\bfB_{\bullet,j} = \argmin_{\substack{\bbeta\in \R^{p}\\ \bbeta_{j}=1}}\limits\min_{c_j\in\R}
\Big\{{\|\bfX\bbeta_{j}- c_j\bfone_n\|}_{2}+\lambda{\|\bbeta\|}_{1}\Big\},\quad j\in [p].
\end{equation}
In addition to being efficiently computable even for large $p$, this estimator has the following appealing property
that makes it preferable, for instance, to the column-wise Lasso \citep{MeinshausenBuhl} and the CLIME \citep{CaiLiuLuo}.
The choice of the parameter $\lambda$ in  (\ref{step:1}-\ref{step:1a}) is scale free: it can be chosen independently of
the noise variance in linear regression (\ref{eq:model}). This fact has been first established by \cite{BelloniChernozhukovWang}
and then further investigated in \citep{SunZhang12,BelloniChernozhukovWang2}. In the context of precision matrix estimation,
this method has been explored\footnote{Although \cite{SunZhang12,SunZhang13} refer to this method as the scaled Lasso, we prefer
to use the original term square-root Lasso coined by \cite{BelloniChernozhukovWang} in order to avoid any possible confusion
with the earlier method of \cite{Stadler1, Stadler2}, for which the term ``scaled Lasso'' has been already employed.}
by \cite{SunZhang13}.

\section{Four estimators of ${\mathbf\phi^*}$}

As mentioned earlier, the aim of this work is to compare different estimators of the vector
$\bphi^*$ based on an initial estimator of the matrix $\bfB^*$. Clearly, the error of the
estimation of $\bfB^*$ impacts the error of the estimation of $\bphi^*$ and, therefore, the
latter is not easy to assess in full generality. In order to gain some insight on the behavior
of various natural estimators, in theoretical results we will consider the ideal situation
where the matrix $\bfB^*$ is estimated without error.

\subsection{Residual variance estimator}
\label{sect:rv}

In view of the regression equation (\ref{eq:model}), a standard and natural method---used, in
particular, by the square-root Lasso of \cite{SunZhang13}---to deduce estimators $\hat\bphi$ and $\hat\bfOmega$
from an estimator $\hat\bfB$ is to set
\begin{equation}\label{step:2}
\hat\phi_{j} = \frac1n{\|(\bfI_n-n^{-1}\bfone_n\bfone^\top_n)\bfX\hat\bfB_{\bullet,j}\|}_{2}^2;\qquad
\hat\bfOmega = \hat\bfB\cdot\bfD_{\hat\bphi}^{-1}.
\end{equation}
Note that the matrix $(\bfI_n-n^{-1}\bfone_n\bfone^\top_n)$ present in this expression is the orthogonal projector in
$\R^n$ onto the orthogonal complement of the linear subspace $\text{Span}(\bfone_n)$ of all constant vectors. The multiplication
by this matrix annihilates the intercept $c_j^*$ in (\ref{eq:model}) and is a standard way of reducing the affine regression to
the linear regression. In what follows, we refer to $\hat\bphi$ defined by (\ref{step:2}) as the residual variance estimator and
denote it by $\bphiRV$. Using the sample covariance  matrix $\bfS_n$, the  residual variance estimator of $\bphi^*$ can be written
as
\begin{equation*}
\phiRV_j  = \hat\bfB_{\bullet,j}^\top \bfS_n \hat\bfB_{\bullet,j}.
\end{equation*}
Note also that if we consider the linear regression model (\ref{eq:model}) conditionally to $\bfX_{\bullet,j^c}$, then
the residual variance estimator of $\phi_j^*$ coincides with the maximum likelihood estimator.


\begin{prop}\label{prop:RVmeanVar}
If $\hat\bfB_{\bullet,j}$ estimates $\bfB^*_{\bullet,j}$ without error, then  the residual variance estimator of $\phi^*_j$
has a quadratic risk equal to $\frac2{n} {\phi_j^*}^2$, that is
\begin{equation*}
\Ex[(\phiRV_j-\phi^*_j)^2] = \frac{2\phi_j^*{}^2}{n}.
\end{equation*}
Furthermore, for every $t>0$, the following bound on the tails of the maximal error holds true:
\begin{equation*}
\prob\bigg(\max_{j\in[p]} \frac{|\phiRV_j-\phi^*_j|}{\phi^*_j} >2\Big(\frac{t+\log p}{n}\Big)^{1/2} +\frac{2(t+\log p ) }{n}\bigg)\le 2e^{-t}.
\end{equation*}
\end{prop}

\begin{proof}
Using equation (\ref{eq:model}) and the assumption  $\hat\bfB_{\bullet,j}=\bfB^*_{\bullet,j}$, we get
\begin{equation}
\phiRV_j = \frac1{n} {\|\bfX\bfB^*_{\bullet,j}\|}_{2}^2 =
\frac{\phi_j^*}{n} {\|\bxi_j\|}_{2}^2.
\end{equation}
Since $\bxi_j$ is a standard Gaussian vector, the random variable
$\zeta = {\|\bxi_j\|}_{2}^2$ is drawn from a $\chi^2_{n}$ distribution.
This implies that $\Ex(\zeta)=n$ and $\Var(\zeta) = 2n$. Therefore,
\begin{equation*}
\Ex[(\phiRV_j-\phi^*_j)^2] =  \Ex\Big[\Big(\frac{\phi_j^*\zeta}{n}-\phi^*_j\Big)^2\Big]
 = \frac{\phi^*_j{}^2}{n^2}\Big(\Var(\zeta)+\big(\Ex(\zeta)-n\big)^2\Big)
 = \frac{2\phi^*_j{}^2}{n}.
\end{equation*}
This completes the proof of the first claim. To prove the second claim, we set $z = t+ \log p$ and use the union bound to get
\begin{align*}
\prob\bigg(\max_{j\in[p]} \frac{|\phiRV_j-\phi^*_j|}{\phi^*_j} >2\sqrt{\frac{z}{n}}  +\frac{2z}{n}\bigg)
        &\le p\max_{j\in [p]} \prob\bigg(\frac{|\phiRV_j-\phi^*_j|}{\phi^*_j} >2\sqrt{\frac{z}{n}} +\frac{2z}{n}\bigg) \\
        &= p\prob\big(|\zeta-n| >2\sqrt{zn} +2z\big).
\end{align*}
The second claim follows from the tail bound of the $\chi^2$ distribution established, for instance, in \cite[Lemma 1]{Laurent2000}.
\end{proof}
Note that in this result, the case of known means $\mu_j^*$ is considered. The case of unknown $\mu_j$ can be handled
similarly, the estimation bias is then $\phi_j^*/n$ and the resulting mean squared error is $(2n-1)\phi_j^*{}^2/n^2$.
One may observe that, as expected, the rate of convergence of the quadratic risk is the usual parametric rate $1/n$ and
that the asymptotic variance is $2\phi_j^*{}^2$.

\subsection{Relaxed maximum likelihood estimator}
\label{sect:ml}

One could expect that the global maximum likelihood estimator of $\bphi^*$ would be better than the
maximum of the conditional likelihood, since it is well known that under proper regularity conditions,
the quadratic risk of the maximum likelihood estimator is the smallest, at least asymptotically. Since
the vectors $\bfX_{k,\bullet}\sim \mathcal N_p(\bmu^*,\bfOmega^*{}^{-1})$ are independent, the log-likelihood
is given by (up to irrelevant additive terms independent of the unknown parameters $\bmu^*$ and $\bfOmega^*$)
\begin{equation}
\mathcal{L}(\bfX|\bmu,\bfOmega) = \frac{n}{2} \log|\bfOmega| - \frac1{2} \sum_{k=1}^n (\bfX_{k,\bullet}-\bmu^\top) \bfOmega
(\bfX_{k,\bullet}-\bmu^\top)^\top.
\end{equation}
Maximizing the log-likelihood with respect to $\bmu\in\R^p$ leads to
\begin{equation}
\max_{\bmu\in\R^p}\mathcal{L}(\bfX|\bmu,\bfOmega) = \frac{n}{2} \Big(\log|\bfOmega| - \trace\big[\bfS_n\bfOmega\big]\Big).
\end{equation}
Recall now that in view of (\ref{eq:normalCorr}), we have $\bfOmega^* = \bfB^* \bfD_{\bphi^*}^{-1}$. Therefore,
the profiled log-likelihood (w.r.t.\ $\bmu$) of $\bfX$ given the parameters $\bfB$ and $\bphi$ is
\begin{equation}\label{eq:PML}
\max_{\bmu\in\R^p}\mathcal{L}(\bfX|\bmu,\bfB,\bphi) = \frac{n}{2} \Big(\log|\bfB|-\sum_{j=1}^p
\big\{\log(\phi_j) +{(\bfS_n\bfB)}_{jj}\phi_j^{-1}\big\}\Big).
\end{equation}
For a given $\bfB$, this profiled  log-likelihood is a decomposable function of $\bphi$ and, therefore, can be
easily maximized with respect to $\bphi$. This leads to
\begin{equation}
\text{arg}\max_{\bphi\in\R_+^p}\max_{\bmu\in\R^p}\mathcal{L}(\bfX|\bmu,\bfB,\bphi) = \big((\bfS_n\bfB)_{jj}\vee 0\big)_{j\in[p]}.
\end{equation}
Thus, when an estimator $\hat\bfB$ of $\bfB^*$ is available, one possible approach for estimating $\bphi^*$ is to set
\begin{equation}\label{RML}
\phiRML_j= (\bfS_n\hat\bfB)_{jj}\vee 0,\qquad j\in[p].
\end{equation}
We call this estimator relaxed maximum likelihood (RML) estimator. It will be clear a little bit later why it is called
relaxed.  The analysis of the risk of the RML estimator is more involved than that of the RV estimator considered
in the previous section. This is due to the truncation at the level $0$. For this reason, the next result does not provide
the precise value of the risk, but just an inequality which is sufficient for our purposes.

\begin{prop}\label{prop:MLmeanVar}
If  $\hat\bfB$ estimates $\bfB^*_{\bullet,j}$ without error, then the risk of the RML estimator of $\phi_j^*$
satisfies  $\Ex[(\phiRML_j-\phi^*_j)^2]\ge \frac1{n} \big( {\phi_j^*}^2 + \phi_j^* \bfSigma^*_{jj} -O(n^{-1/2})\big)$.
\end{prop}

Before providing the proof of this result, let us present a brief discussion. Note that in view of (\ref{eq:model}),
$\bfSigma^*_{jj}$ is always not smaller than $\phi_j^*$. Furthermore, $\bfSigma^*_{jj}>\phi_j^*$ if $\bfB^*_{j^c,j}$
has at least one nonzero entry.
Therefore, the last proposition, combined with Proposition~\ref{prop:RVmeanVar}, establishes that the
residual variance estimator has an asymptotic variance which is  smaller (and, in many cases, strictly smaller)
than the asymptotic variance of the maximum likelihood estimator. At a first sight, this is very surprising and seems to be
in contradiction with the well established theory \citep{LeCamBook,IbragimovBook} of asymptotic efficiency of the maximum
likelihood estimator for regular models. Our explanation of this inefficiency of $\phiRML_j$ is that it is not really
the maximum likelihood estimator. It maximizes the likelihood, certainly, but not over the correct set of parameters.
Indeed, when we  defined the RML we neglected an important property of the vector $\bphi^*$: the fact that
$\bfB^*\bfD_{\bphi^*}^{-1}=\bfD_{\bphi^*}^{-1}\bfB^*{}^\top$ (this follows from the symmetry of $\bfOmega^*$).
Ignoring this constraint allowed us to get a tractable optimization problem but caused the loss of the (asymptotic)
efficiency of the estimator. This also explains why we call $\bphiRML$ \textit{relaxed} maximum likelihood estimator.

\begin{proof}[Proof of Proposition \ref{prop:MLmeanVar}]
Since $\bmu^*$ is assumed to be known and equal to zero, according to (\ref{eq:model}), we have
\begin{equation*}
(\bfS_n\hat\bfB)_{jj} = \frac1n \bfX_{\bullet,j}^\top \bfX\bfB^*_{\bullet,j} = \frac{\phi_j^*{}^{1/2}}n \bfX_{\bullet,j}^\top\bxi_j
=\frac{\phi_j^*{}^{1/2}}n \big(-\bfX_{\bullet,j^c}\bfB^*_{j^c,j}+\phi_j^*{}^{1/2}\;\bxi_j\big)^\top\bxi_j.
\end{equation*}
Denoting $\eta_1=-\bxi_j^\top\bfX_{\bullet,j^c}\bfB^*_{j^c,j}$, we get
$(\bfS_n\hat\bfB)_{jj} = \frac1n({\phi_j^*{}^{1/2}}\;\eta_1+ {{\phi_j^*}} \;\|\bxi_j\|_2^2)$. Furthermore,
it follows from (\ref{eq:model}) that $\Ex[(\bfX_{k,j^c}\bfB^*_{j^c,j})^2]= \bfSigma_{jj}^*-\phi_j^*$ for each $k$. Since, in addition,
for different $k$s the random variables $\bfX_{k,j^c}\bfB^*_{j^c,j}$ are independent, centered and Gaussian, we get
that---in view of the independence of $\bxi_j$ and $\bfX_{\bullet,j^c}$---the conditional distribution of
$\eta_1$ given $\bxi_j$ is Gaussian with zero mean and variance  $\|\bxi_j\|_2^2(\bfSigma^*_{jj}-\phi_j^*)$. Hence,
the random variable $\eta = \eta_1/(\|\bxi_j\|_2(\bfSigma^*_{jj}-\phi_j^*)^{1/2})$ is standard Gaussian, independent
of $\|\bxi_j\|_2^2$ and
\begin{equation*}
(\bfS_n\hat\bfB)_{jj} = \frac{\sqrt{\phi_j^*\|\bxi_j\|_2^2(\bfSigma^*_{jj}-\phi_j^*)}}n \;\eta+
\frac{{\phi_j^*}}n \;\|\bxi_j\|_2^2.
\end{equation*}
This relation readily implies that $\Ex[(\bfS_n\hat\bfB)_{jj}]= \phi_j^*$ and
\begin{equation*}
\Ex[((\bfS_n\hat\bfB)_{jj}-\phi_j^*)^2]=\Var[(\bfS_n\hat\bfB)_{jj}] = \frac{\bfSigma_{jj}^*\phi_j^*+\phi_j^*{}^2}{n}.
\end{equation*}
Furthermore, for the fourth moment, we have
\begin{align*}
\Ex[((\bfS_n\hat\bfB)_{jj}-\phi_j^*)^4] &\le  \frac{8\phi_j^*{}^2(\bfSigma^*_{jj}-\phi_j^*)^2}{n^4}\Ex[\|\bxi_j\|_2^4]\Ex[\eta^4]
+\frac{8\phi_j^*{}^4}{n^4} \;\Ex[(\|\bxi_j\|_2^2-n)^4]\nonumber\\
&\le  \frac{72\phi_j^*{}^2(\bfSigma^*_{jj}-\phi_j^*)^2}{n^2}+\frac{8\phi_j^*{}^4}{n^4} (60n+12n^2).
\end{align*}
To analyze the truncated estimator, we set $\zeta = (\bfS_n\hat\bfB)_{jj}$. Then $\phiRML_j = \zeta\cdot\mathds 1(\zeta>0)$ and hence,
\begin{align*}
\Ex[(\phiRML_j -\phi_j^*)^2]
		&=\Ex[(\zeta -\phi_j^*)^2\mathds 1(\zeta>0)]+\phi_j^*{}^2\prob(\zeta\le 0) \\
		&=\Ex[(\zeta -\phi_j^*)^2]-\Ex[(\zeta -\phi_j^*)^2\mathds 1(\zeta\le 0)]+\phi_j^*{}^2\prob(\zeta\le 0) \\
		&\ge \Ex[(\zeta -\phi_j^*)^2]-\Ex[(\zeta -\phi_j^*)^4]^{1/2}\prob(\zeta\le 0)^{1/2}.		
\end{align*}
We have already computed the first expectation in the right-hand side, as well as upper bounded the second one. Let
us show that the probability of the event $\zeta\le 0$ goes to zero as $n$ increases to $\infty$. This follows from the
Tchebychev inequality, since $\prob(\zeta\le 0) = \prob(\phi_j^*-\zeta\ge \phi_j^*)\le \Var[\zeta]/\phi_j^*{}^2=O(1/n)$.
This completes the proof of the proposition.
\end{proof}

\subsection{MLE taking into account the symmetry constraints}
\label{sect:mls}

As we have seen in previous sections, the relaxed maximum likelihood estimator is suboptimal; in particular,
it is less accurate than the residual variance estimator. To check that this lack of efficiency is indeed due
to the relaxation of the symmetry constraints, we propose here to analyze the constrained maximum likelihood
estimator in the following idealized set-up. We will consider, as in Propositions~\ref{prop:RVmeanVar} and
\ref{prop:MLmeanVar}, that $\hat\bfB$ estimates $\bfB^*$ without error, and that\footnote{This assumption
will be relaxed later in this subsection.} there is a column $\bfB^*_{\bullet,i}$ in $\bfB^*$ such that all
the elements of $\bfB^*_{\bullet,i}$ are different from zero. Without loss of generality, we suppose that
$i=1$ and, consequently, for every $j\in[p]$, we have $\bfB^*_{j1} \ne 0$ which is equivalent to $\omega^*_{j1} \ne 0$.
Therefore, the symmetry constraint $\bfB^*\bfD_{\bphi^*}^{-1}=\bfOmega^*=\bfOmega^*{}^\top=\bfD_{\bphi^*}^{-1}\bfB^*{}^\top$
implies that $\bfD_{\bphi^*}\bfB^*=\bfB^*{}^\top\bfD_{\bphi^*}$ and, in particular, that
\begin{equation*}
 \bfB^*_{1j}\phi_1^* = \bfB^*_{j1}\phi_j^* ,\qquad \forall j\in[p] .
\end{equation*}
This relation entails that in the case of known matrix $\bfB^*$ and unknown vector $\bphi^*$, only the first
entry of $\bphi^*$ needs to be estimated, all the remaining entries can be computed using the first one
by the formula $\hat\phi_j = (\bfB^*_{1j}/\bfB^*_{j1})\hat\phi_1$.

\begin{prop}\label{prop:ML1sym}
Under the assumption that the rows of $\bfX$ are i.i.d.\ Gaussian vectors with precision matrix $\bfOmega^* = \bfB^*\bfD_{\bphi^*}^{-1}$,
the maximum likelihood estimator of $\bphi^*$ is defined by
\begin{equation}\label{phiML}
\phiML_j = \frac1{p} (\bfB^*_{1j}/\bfB^*_{j1})\trace \big(\bfS_{n} \bfB^* \bfD_{\bfB^*_{\bullet,1}} \bfD_{\bfB^*_{1,\bullet}}^{-1}\big).
\end{equation}
The quadratic risk of this estimator is given by
\begin{equation}\label{riskML1}
\Ex[(\phiML_j-\phi_j^*)^2] = \frac2{np}\;\phi_j^*{}^2.
\end{equation}
Furthermore, for every $t>0$, the following bound on the tails of the maximal error holds true:
\begin{equation*}
\prob\bigg(\max_{j\in[p]} \frac{|\phiML_j-\phi^*_j|}{\phi^*_j} >2\Big(\frac{t+\log p}{np}\Big)^{1/2} +\frac{2(t+\log p ) }{np}\bigg)\le 2e^{-t}.
\end{equation*}
\end{prop}
\begin{proof}
To ease notation, we denote by $\bfD^*$ the diagonal matrix whose $j$th element is $\bfB^*_{j,1} / \bfB^*_{1,j}$.
Then, applying (\ref{eq:PML}) for a given $\bfB^*$, the profiled Gaussian log-likelihood can be written as
\begin{align*}
 \max_{\bmu\in\R^p}\mathcal{L}(\bfX|\bmu,\bfB^*,\bphi) = \frac{n}{2}\log|\bfB^*|-\frac{n}{2} \sum_{j=1}^p \big\{\log(\phi_j) +
		(\bfS_n\bfB^*)_{jj}\phi_{j}^{-1}\big\}.
\end{align*}
The goal is to maximize the right-hand side over all the vectors $\bphi\in\R^p$ such that $\bfB^*\bfD_{\bphi}^{-1}$ is a valid precision
matrix.

Let us first check that under the conditions of the proposition, for $\bfB^*\bfD_{\bphi}^{-1}$ to be a valid precision
matrix it is necessary and sufficient that $\phi_1>0$ and $\phi_j = (\bfB^*_{1j}/\bfB^*_{j1})\phi_1$ for every $j\in[p]$.
The necessary part follows from that fact that a precision matrix is symmetric and positive-semidefinite, which entails
that $(\bfB^*\bfD_{\bphi}^{-1})_{1j} = (\bfB^*\bfD_{\bphi}^{-1})_{j1}$  and $(\bfB^*\bfD_{\bphi}^{-1})_{jj} = \bphi_j^{-1}>0$.
Therefore, $\phi_j = (\bfB^*_{1j}/\bfB^*_{j1})\phi_1$ and $\phi_1>0$. To check the sufficient part, we remark that if
$\bphi$ satisfies $\phi_j = (\bfB^*_{1j}/\bfB^*_{j1})\phi_1$ with $\phi_1>0$, then $\bfB^*\bfD_{\bphi}^{-1} =
(\phi^*_1/\phi_1)\bfB^*\bfD_{\bphi^*}^{-1}=(\phi^*_1/\phi_1)\bfOmega^*$. This implies that $\bfB^*\bfD_{\bphi}^{-1}$
is symmetric and positive-semidefinite, hence a valid precision matrix.

The maximum likelihood estimator $\bphiML$ is thus given by
\begin{equation*}
\bphiML\in\text{arg}\min_{\substack{\bphi\in\R^p_+\\ \phi_j = (\bfB^*_{1j}/\bfB^*_{j1})\phi_1}} \sum_{j=1}^p \big\{\log(\phi_j) +
		(\bfS_n\bfB^*)_{jj}\phi_{j}^{-1}\big\},
\end{equation*}
which leads to $\phiML_1\in\text{arg}\min_{\phi_1>0} \big\{p\log(\phi_1) +
\phi_{1}^{-1}\sum_j (\bfS_n\bfB^*)_{jj}\bfB_{j1}^*/\bfB^*_{1j}\big\}$. The cost function of the last minimization problem is convex, since
we have $(\bfS_n\bfB^*)_{jj}\bfB_{j1}^*/\bfB^*_{1j} = (\bfS_n\bfB^*)_{jj}\phi_{1}^*/\phi^*_{j}=\phi_{1}^*(\bfS_n\bfOmega^*)_{jj}$. This implies that
\begin{equation*}
\sum_j(\bfS_n\bfB^*)_{jj}\bfB_{j1}^*/\bfB^*_{1j} =\phi_{1}^*\trace(\bfS_n\bfOmega^*)=\phi_{1}^*\trace(\bfOmega^*{}^{1/2}
\bfS_n\bfOmega^*{}^{1/2})\ge 0.
\end{equation*}
The aforementioned cost function is continuously differentiable and convex, its minimum is attained at
the point where the derivative vanishes, which provides $\phiML_1 = \frac1p\sum_j (\bfS_n\bfB^*)_{jj}\bfB_{j1}^*/\bfB^*_{1j}$.
Combining with the relation $\phiML_j = (\bfB^*_{1j}/\bfB^*_{j1})\phiML_1$, this leads to (\ref{phiML}).

To check (\ref{riskML1}), we start by noting that
\begin{equation*}
\phiML_j = \frac1p \phi^*_j \trace(\bfS_n\bfOmega^*) = \frac1{np} \phi^*_j \trace(\bfX^\top\bfX\bfSigma^*{}^{-1}).
\end{equation*}
Using the well-known commutativity property of the trace operator and setting $\bfY = \bfSigma^*{}^{-1/2}\bfX^\top$,
we get $\trace(\bfX^\top\bfX\bfSigma^*{}^{-1})=\trace(\bfY^\top\bfY)$. Since $\bfX$ has
iid rows drawn from a $\mathcal N_p(0,\bfSigma^*)$ distribution, $\bfY$ has iid columns drawn from $\mathcal N_p(0,\bfI_p)$
distribution. Hence, the random variable $\trace(\bfY^\top\bfY) = \sum_{j\in[p],k\in[n]}\bfY_{jk}^2$ is distributed according to
$\chi^2_{np}$ distribution. This readily implies that $\phiML_j$ is an unbiased estimator of $\phi_j^*$ and, therefore, its
quadratic risk  coincides with its  variance  and is given by (\ref{riskML1}).

The proof of the last claim of the proposition is very similar to that of the second claim
of Proposition~\ref{prop:RVmeanVar}.
\end{proof}

Assuming that there exists $i\in[p]$ such that for any $j\in[p]$, $\omega^*_{ij} \ne 0$, put differently that
the $i$-th node of the graph $\mathscr G^*$  is connected by an edge to any other node is quite restrictive.
Among other implications, it entails that the graph $\mathscr G^*$ is connected which might be a strong assumption.
It is therefore useful to adapt what precedes to the case where the graph $\mathscr G^*$ has more than one
connected component. The rest of this subsection is devoted to the description of this adaptation.

We note $\mathscr C$ the set of the connected components of the graph $\mathscr G^*$. Each connected component
$c \in \mathscr C$ is a subset of vertices of $\mathscr G^*$ whose cardinality is denoted by $p_c$. Clearly,
the sum of $p_c$ over all $c\in \mathscr C$ equals $p$. For two vertices $i$ and $j$, we will write $i \sim_{\mathscr G^*\!} j$
for indicating that they belong to the same connected component. Thus, each connected component is a class of equivalence
with respect to the relation $\sim_{\mathscr G^*\!}$.
Let $i\sim_{\mathscr G^*\!} j$ be two vertices from $c\in \mathscr C$ and let $C_{ji}$ be a path connecting these two
vertices,\textit{i.e.}, $C_{ji}$ is a sequence of $q$ distinct vertices $\{ v_1 , \ldots , v_q \}$ such that $v_1 = j$,
$v_q = i$, $q \le p_c$ and each pair $(v_h,v_{h+1})$ is connected by an edge in $\mathscr G^*$. Recall that the symmetry
of the precision matrix $\bfOmega^* = \bfB^*\bfD^{-1}_{\bphi^*}$ implies that $\bfB^*_{v_h,v_{h+1}}\phi^*_{v_h} =
\bfB^*_{v_{h+1},v_h}\phi^*_{v_{h+1}}$ for every $h\in[q-1]$. This readily yields
\begin{equation*}
 \phi^*_{j} = \phi^*_{i} \prod_{1 \le h < q}  \big(\bfB^*_{v_{h+1},v_h}/\bfB^*_{v_h,v_{h+1}}\big).
\end{equation*}
To ease notation, we introduce the $p \times p$ diagonal matrix $\bfDelta_j^*$  the diagonal entries of which
are defined by
\begin{equation}\label{delta}
 (\bfDelta_j^*)_{ii} = \mathds 1(i \sim_{\mathscr G^*}\! j)
\prod_{1 \le h < q} \big(\bfB^*_{v_{h+1},v_h}/\bfB^*_{v_h,v_{h+1}}\big),
\end{equation}
where $\{v_1,\ldots,v_q\} = C_{ji}$ is any path connecting $j$ to $i$ in $\mathscr G^*$. With this notation,
$\phi^*_j =(\bfDelta_j^*)_{ii}\phi^*_i$. One can reproduce the arguments of the proof of Proposition~\ref{prop:ML1sym}
to check that the maximum likelihood estimator of $\bphi^*$, if $\bfB^*$ is known (and therefore so is $\bfDelta_j^*$),
is defined by
\begin{equation}\label{phiMLS}
\phiML_j = \frac1{p_c} \trace \big(\bfDelta_j^* \bfS_{n} \bfB^*\big),
\end{equation}
for $j$ belonging to the connected component $c$.

Comparing the results of Propositions~\ref{prop:RVmeanVar}, \ref{prop:MLmeanVar} and \ref{prop:ML1sym}, we observe that
the RV-estimator outperforms the RML estimator, but---at least in the case where there is a column in $\bfB^*$ which has
only nonzero entries---they are both dominated by the maximum likelihood estimator that takes advantage of the symmetry
constraints. Furthermore, using the same type of arguments as those of Proposition~\ref{prop:ML1sym}, one can check that if the
vertex $j$ of the graph $\mathscr G^*$ belongs to a connected component of cardinal $p_c$ then the risk of the MLE in the ideal
case of known $\bfB^*$ is equal to $\frac2{np_c}\phi_j^*{}^2$. This shows that in the ideal case the MLE systematically
outperforms the widely used residual variance estimator, and the gain in the risk may be huge for vertices belonging to
large connected components. On the other extreme, all the three estimators discussed in the previous section coincide when
the matrix $\bfB^*$ is diagonal.

In order to apply equation (\ref{phiMLS}) for estimating $\bphi^*$ when an estimator $\hat\bfB$ of $\bfB^*$ is available,
we need to construct an estimator $\hat{\mathscr G}$ of the graph $\mathscr G^*$. We propose here an original approach for
deriving $\hat{\mathscr G}$ from $\hat\bfB$. It is based on the observation that $\bfB^*_{ij}\bfB^*_{ji} =
{\omega^*_{ij}}^2/(\omega_{ii}^*\omega_{jj}^*)^2$, the square of the partial correlation between the $i$-th and $j$-th
variables. As mentioned earlier, this quantity is always between 0 and 1 and provides a convenient rule of selection for
the edges to keep in the graph. More precisely, we connect $i$ to $j$ if the estimated squared partial correlation
$\hat\bfB_{ij}\hat\bfB_{ji}$ is larger than a prescribed threshold $t\in (0,1)$. In our implementation, we chose (somewhat
arbitrarily) the threshold $t=0.01 \wedge n^{-1/2}$.

Note that when $\bfB^*$ is replaced by an estimator, the right-hand side of (\ref{phiMLS}) is not necessarily invariant
with respect to the choice of the path connecting $i$ to $j$. Therefore, even when $\hat\bfB$ and $\hat{\mathscr G}$ are
fixed, if $\hat{\mathscr G}$ contains loops there are different ways of estimating $\bphi^*$ based on (\ref{phiMLS})
depending on how the paths are chosen. We have tried two possible approaches: the minimum spanning tree and the shortest
path tree based on the following weight function\footnote{A weight equal to zero corresponds to the absence of edge.}
defined on the edges:
\begin{align*}
\bfW_{ij} = \left\lbrace \begin{array}{ll}
                  \exp\big(- \hat\bfB_{ij}\hat\bfB_{ji}\big)
									\mathds 1(\hat\bfB_{ij}\hat\bfB_{ji} > t), & \qquad
									\text{for } i \ne j , \\
                  0, &\qquad \text{otherwise.}
                 \end{array}
	  \right.
\end{align*}
Combining these ingredients, we get the algorithm summarized in Algorithm~\ref{algo:SPT}.
\begin{algorithm}
   \caption{Estimator $\bphiML$  based on shortest path trees or minimum spanning trees}
   \label{algo:SPT}
\begin{algorithmic}
\normalsize
\STATE {\bfseries Input:} matrices $\bfX$ and $\hat\bfB$, threshold $t$.
\STATE {\bfseries Output:} vector $\bphiML$.
\STATE {\tt 1:} compute the matrix of weights $\bfW$.
\STATE {\tt 2:} initialize $k$ to 1.
\REPEAT
\STATE {\tt 3:} choose the node with the largest degree as root.
\STATE {\tt 4:} compute the shortest path tree (or the minimum spanning tree)
$\mathscr T_k$ from the chosen root.
\STATE {\tt 5:} estimate $\bphiML$'s elements related to $\mathscr T_k$
using Eq.\ (\ref{phiMLS})
\STATE {\tt 6:} remove all the nodes of the tree $\mathscr T_k$ from the initial graph.
\STATE {\tt 7:} increment $k$.
\UNTIL{graph is empty}
\end{algorithmic}
\end{algorithm}

The rationale behind the foregoing definition of the weights and the use of the minimum spanning tree or
shortest path tree algorithm is to favor the paths that are short and contain edges corresponding to large
(in absolute value) partial correlations. The aim is to reduce the risk of propagating the estimation error
of $\hat\bfB$. We have implemented both versions of the algorithm and have observed that the version using
the minimum spanning tree leads to better results. More details on the implementation and computational
complexity are given in the next section.

\begin{table}[ht]
\begin{tabular}{|*{10}{c|}}
\toprule
$p$ & \multicolumn{3}{c|}{30} & \multicolumn{3}{c|}{60} & \multicolumn{3}{c|}{90}\\
\midrule
$n$ & 200 & 800 & 2000 & 200 & 800 & 2000 & 200 & 800 & 2000 \\
\midrule
\multicolumn{10}{|c|}{\bf $\bfB^*$ estimated by square-root Lasso} \\
\midrule
RV & {\bf 0.883} & {\bf 0.399} & {\bf 0.224} & {\bf 1.425} & {\bf 0.649} & {\bf 0.374 } & {\bf 1.849} & {\bf 0.853} & {\bf 0.495} \\
 & {\bf (.077)} & {\bf (.036)} & {\bf (.016)} & {\bf (.075)} & {\bf (.030)} & {\bf (.022)} & {\bf (.085)} & {\bf (.029)} & {\bf (.019)} \\
\midrule
RML & 1.356  & 0.786  & 0.532  & 2.114  & 1.234  & 0.841  &  2.705 &  1.590 & 1.086 \\
   & (.079) & (.040) & (.017) & (.082) & (.032) & (.022) & (.090) & (.029) & (.019) \\
\midrule
SML & 1.476  & 0.805  & 0.548  & 2.388 & 1.250  & 0.852  & 3.104& 1.608  & 1.096\\
    & (.098) & (.040) & (.018) &  (.164) & (.032) & (.021) & (.188) & (.028) & (.020) \\
\midrule
PML & 1.371  & 0.792  & 0.539 & 2.134  & 1.236  & 0.846  & 2.728 & 1.593  & 1.089\\
    & (.079) & (.041) &(.017) &(.078) &  (.032) & (.021) & (.091) & (.030) & (.019) \\
\midrule
\multicolumn{10}{|c|}{\bf $\bfB^*$ estimated by square-root Lasso followed by OLS} \\
\midrule
RV & {\bf 0.726 } & {\bf 0.340} & {\bf 0.241 } & {\bf 1.088 } & {\bf 0.616 } & {\bf 0.354 } & {\bf 1.365} & {\bf 0.854 } & {\bf 0.443 } \\
& {\bf (.079)} & {\bf (.045)} & {\bf  (.016)} & {\bf (.076)} & {\bf (.051)} & {\bf(.020)} & {\bf (.080)} & {\bf (.046)} & {\bf (.018)} \\
\midrule
RML & {\bf 0.726 } & {\bf 0.340 } & {\bf 0.241} & {\bf 1.088} & {\bf 0.616} & {\bf 0.354} & {\bf 1.365} & {\bf 0.854} & {\bf 0.443} \\
 & {\bf (.079)} & {\bf (.045)} & {\bf (.016)} & {\bf (.076)} & {\bf (.051)} & {\bf (.020)} & {\bf (.080)} & {\bf (.046)} & {\bf (.018)} \\
\midrule
SML & 0.807  & 0.440  & 0.280  & 1.193 & 0.793 & 0.381 & 1.557  & 1.116  & 0.468 \\
 &  (.082) & (.058) &(.018) &  (.088) &  (.066) & (.018) & (.170) &  (.089) &  (.018) \\
\midrule
PML & 0.737  & 0.419 & 0.302  & 1.095 & 0.722 & 0.405 & 1.377 & 0.984  & 0.494 \\
 &  (.074) & (.051) & (.018) &  (.071) &  (.052) & (.019) &(.081) & (.044) & (.019) \\
\midrule
\multicolumn{10}{|c|}{\bf $\bfB^*$ is estimated without error } \\
\midrule
RV & 0.263 & 0.132 & 0.081 & 0.370 & 0.179 & 0.115 & 0.455 & 0.222 & 0.143 \\
 & (.034) & (.017) & (.012) & (.038) & (.017) & (.008) & (.038) & (.019) & (.012) \\
 \midrule
RML & 0.322  & 0.165 & 0.104& 0.463 & 0.227  & 0.144  & 0.562 & 0.280 & 0.178  \\
 & (.042) &(.018) & (.013) &  (.038) & (.022) & (.011) & (.040) & (.020) & (.015) \\
 \midrule
SML & {\bf 0.043 } & {\bf 0.024 } & {\bf 0.010} & {\bf 0.042} & {\bf 0.018} & {\bf 0.011} & {\bf 0.042} & {\bf 0.015} & {\bf 0.010} \\
 & {\bf  (.030)} & {\bf (.018)} & {\bf (.010)} & {\bf (.030)} & {\bf (.014)} & {\bf (.009)} & {\bf(.037)} & {\bf (.013)} & {\bf (.007)} \\
 \midrule
PML & 0.079  & 0.043  & 0.023  & 0.107  & 0.049 & 0.030  & 0.128 & 0.059  & 0.039 \\
 & (.025) &  (.015) & (.007) &  (.028) & (.012) &  (.007) & (.027) &  (.011) & (.007) \\
\bottomrule
\end{tabular}
\caption{Performance of the estimators of diagonal elements of the precision matrix in Model 1. The number of replications in each case
is $R=50$. More details on the experimental set-up are presented in Section~\ref{ssec:exp}.}
\label{tab:1}
\end{table}

\subsection{Penalized maximum likelihood estimation}
\label{sect:pml}

We have seen that enforcing symmetry constraints is beneficial when the matrix $\hat\bfB$ has a small error, but raises
intricate issues related to the graph estimation and, more importantly, path selection in the graph. A workaround to this issue
is to replace the hard constraints by a penalty term that measures the degree of violation of the constraints. This provides an
intermediate solution between the SML and the RML. More precisely, we propose a penalized maximum likelihood (PML) estimator of
$\bphi^*$ defined by
\begin{equation}
 \bphiPML \in \argmin_{\bphi \in (0,1]^p} \bigg\{\sum_{j=1}^p \{ \log(\phi_j) + (\bfS_n \hat\bfB)_{j,j} \phi_j^{-1} \} + \kappa
\sum_{i<j \atop \hat\bfB_{ji}\hat\bfB_{ij}>t} \frac{(\hat\bfB_{ji} \phi_i^{-1} - \hat\bfB_{ij} \phi_j^{-1})^2}{\hat\bfB_{ij}^2 +
\hat\bfB_{ji}^2}\bigg\},
 \label{eq:phiPML}
\end{equation}
where $\kappa>0$ is a tuning parameter responsible for the trade-off between the likelihood and the constraint violation.
The choice $\kappa=\infty$ corresponds to enforcing the symmetry constraints: its main shortcoming is that the feasible set
might very well be empty. On the other extreme, when $\kappa=0$, PML coincides with the RML.
The PML estimator also coincides with the previous ones if $\hat\bfB$ is known to be diagonal.

Note that the parameter $t$ appearing in the penalty term of the PML plays the same role as the one used in the SML. The definition
of the feasible set in the above optimization problem is justified by the fact that we assume all the individual variances of the
features to be equal to one. In other terms, the assumption $\Var(\bfX_{1,j})=1$ in (\ref{eq:model}) implies that $\phi_j^*\le 1$.
Making the change of variable $\bv=(1/\phi_j)_{j \in [p]}$, the optimization problem of Eq.\ (\ref{eq:phiPML}) becomes convex with
the feasible set $\bv\in[1+\infty)^p$ and the objective function:
\begin{equation}\label{eq:f}
f(\bv)=\sum_{j=1}^p \big\{ -\log(v_j) + (\bfS_n \hat\bfB)_{j,j} v_j \big\} + \kappa
\sum_{i<j \atop \hat\bfB_{ji}\hat\bfB_{ij}>t} \frac{(\hat\bfB_{ji} v_i - \hat\bfB_{ij} v_j)^2}{\hat\bfB_{ij}^2 +
\hat\bfB_{ji}^2}
\end{equation}
Furthermore, if we restrict the feasible set to $\bv \in \mathcal{V} = [1,n^{1/2}]^p$, the problem becomes strongly convex.
In addition, on this restricted feasible set the gradient of the objective function is Lipschitz-continuous.

It is possible to use the standard steepest gradient descent algorithm with a fixed step-size for efficiently approximating
the solution $\bphiPML$. Indeed, in the optimization problem \eqref{eq:f}, if $\nabla f$ is Lipschitz-continuous with constant $L<\infty$
and strongly convex with constant $l>0$, the gradient descent algorithm with a constant step-size $t = 2/( l + L)$ converges at a linear
rate (see \cite{Nesterov2004} for a detailed proof). Note that the convergence rate depends on ${L}/{l}$ which is an upper bound on the
condition number of the Hessian matrix $\nabla^2 f(\bv)$; this ratio should not be too high for the algorithm to converge fast.
Unfortunately, the values of $l$ and $L$ that we manage to obtain in our problem are far too loose. That is why we resort to a steepest
descent algorithm with adaptive step-size and scaled descent direction $- \nabla f(\bv_h) / \| \nabla f(\bv_h) \|_2$.
More details on the implementation are provided in Section~\ref{ssec:impl}.

\section{Experimental evaluation}

In this section, we describe the experimental set-up and report the results of the numerical experiments performed
on synthetic data-sets. We also provide detailed explanation of the implementation used for the symmetry-enforced and the
penalized maximum-likelihood estimators. A companion R package called \texttt{DESP} (for estimation of Diagonal Elements of Sparse
Precision-matrices) is created  and uploaded on CRAN\footnote{\url{http://cran.r-project.org/web/packages/DESP/index.html}}.

\begin{table}[ht]
\begin{tabular}{|*{10}{c|}}
\toprule
$p$ & \multicolumn{3}{c|}{30} & \multicolumn{3}{c|}{60} & \multicolumn{3}{c|}{90}\\
\midrule
$n$ & 200 & 800 & 2000 & 200 & 800 & 2000 & 200 & 800 & 2000 \\
\midrule
\multicolumn{10}{|c|}{\bf $\bfB^*$ estimated by square-root Lasso} \\
\midrule
RV & {\bf 0.400 } & {\bf 0.125} & {\bf 0.070} & {\bf 0.632} & {\bf 0.174} & {\bf 0.094} & {\bf 0.821 } & {\bf 0.215} & {\bf 0.113} \\
 & {\bf (.059)} & {\bf (.020)} & {\bf(.009)} & {\bf (.047)} & {\bf (.023)} & {\bf(.011)} & {\bf (.051)} & {\bf  (.020)} & {\bf  (.012)} \\
\midrule
RML & 1.048 & 0.508  & 0.320  & 1.644 & 0.780 & 0.491  & 2.120  & 0.997  & 0.626  \\
 & (.061) & (.020) & (.015) & (.048) &(.023) & (.014) & (.053) & (.023) & (.014) \\
\midrule
SML & 1.334  & 0.539 & 0.340  & 2.246  & 0.824  & 0.520 & 3.243  & 1.047& 0.653 \\
 &  (.221) & (.028) &(.018) & (.277) & (.039) & (.023) & (.516) & (.034) & (.020) \\
\midrule
PML & 1.130  & 0.530 & 0.333 & 1.790 & 0.813 & 0.508 & 2.311 & 1.036 & 0.645  \\
 & (.068) & (.020) & (.016) &(.049) & (.026) & (.016) & (.054) & (.024) & (.015) \\
\midrule
\multicolumn{10}{|c|}{\bf $\bfB^*$ estimated by square-root Lasso followed by OLS} \\
\midrule
RV & {\bf 0.247 } & 0.101  & 0.065  & {\bf 0.322 } & 0.129 & 0.081  & {\bf 0.381} & 0.150 & 0.095  \\
 & {\bf (.053)} & (.015) & (.009) & {\bf (.057)} &  (.019) &  (.007) & {\bf (.061)} &  (.017) &  (.009) \\
\midrule
RML & {\bf 0.247} & 0.101  & 0.065  & {\bf 0.322 } & 0.129  & 0.081 & {\bf 0.381 } & 0.150  & 0.095 \\
 & {\bf  (.053)} &  (.015) & (.009) & {\bf  (.057)} & (.019) & (.007) & {\bf  (.061)} & (.017) & (.009) \\
\midrule
SML & 0.329  & {\bf 0.096} & 0.065  & 0.622  & 0.129  & {\bf 0.076 } & 0.882 & 0.147 & 0.090  \\
 & (.107) & {\bf (.016)} &(.010) &  (.299) & (.021) & {\bf (.010)} &  (.501) &  (.020) & (.011) \\
\midrule
PML & {\bf 0.247 } & 0.098  & {\bf 0.064 } & 0.337  & {\bf 0.125 } & 0.077  & 0.441  & {\bf 0.142 } & {\bf 0.089} \\
 & {\bf (.068)} & (.017) & {\bf (.011)} & (.075) & {\bf (.021)} & (.009) & (.101) & {\bf (.017)} & {\bf (.011)} \\
\midrule
\multicolumn{10}{|c|}{\bf $\bfB^*$ is estimated without error } \\
\midrule
RV & 0.204& 0.101 & 0.065 & 0.258 & 0.129  & 0.081  & 0.300  & 0.149 & 0.095 \\
 & (.032) & (.015) &(.008) & (.033) & (.019) &(.007) &(.030) &(.015) & (.009) \\
\midrule
RML & 0.280  & 0.136  & 0.086  & 0.354 & 0.177  & 0.113  & 0.429 & 0.214 & 0.135  \\
RML & (.038) & (.017) & (.011) &(.032) &  (.019) &(.010) &(.038) & (.021) & (.012) \\
\midrule
SML & {\bf 0.033 } & {\bf 0.012 } & {\bf 0.008 } & {\bf 0.024 } & {\bf 0.012 } & {\bf 0.008 } & {\bf 0.027 } & {\bf 0.011 } & {\bf 0.007 } \\
SML & {\bf(.022)} & {\bf (.008)} & {\bf (.007)} & {\bf (.017)} & {\bf (.009)} & {\bf (.006)} & {\bf (.019)} & {\bf (.008)} & {\bf (.006)} \\
\midrule
PML & 0.065 & 0.027 ( & 0.019  & 0.065  & 0.031  & 0.021  & 0.073  & 0.035  & 0.023  \\
PML & (.021) &(.011) & (.006) & (.020) &(.009) & (.006) & (.022) & (.010) & (.006) \\
\bottomrule
\end{tabular}
\caption{Performance of the estimators of diagonal elements of the precision matrix in Model 2. The number of replications in each case
is $R=50$. More details on the experimental set-up are presented in Section~\ref{ssec:exp}.}
\end{table}

\subsection{Experiments on synthetic datasets}\label{ssec:exp}

We conducted a comprehensive experimental evaluation of the accuracy of different estimates of diagonal elements of the
precision matrix. In order to cover as many situations as possible, we used in experiments our six different forms of precision matrices
along with various values for $n$ and $p$. In each configuration, we considered several methods of estimating the matrix
$\bfB^*$.

Let us first describe  in a precise manner the precision matrices used in our experiments. It is worthwhile to underline here that all the precision
matrices are normalized in such a way that all the diagonal entries of the corresponding covariance matrix
$\bfSigma^* = (\bfOmega^*)^{-1}$ are equal to one. To this end, we first define a $p\times p$ positive semidefinite matrix $\bfA$
and then set $\bfOmega^* = (\diag(\bfA^{-1}))^{\frac1{2}} \bfA (\diag(\bfA^{-1}))^{\frac1{2}}$. The matrices $\bfA$ used in the six models
for which the experiments are carried out are defined as follows.

\begin{description}\itemsep=4pt
\item[{\bf Model 1:}] $\bfA$ is a Toeplitz matrix with the entries $\bfA_{ij} = 0.6^{|i-j|}$ for any $i,j \in [p]$.

 \item[{\bf Model 2:}] We start by defining a $p\times p$ pentadiagonal matrix with the entries
 \begin{equation*}
 \bar\bfA_{ij} = \left\lbrace \begin{array}{cl}
                                                       1 &, \text{ for } |i-j| = 0, \\
                                                       -1/3 &, \text{ for } |i-j| = 1, \\
                                                       -1/10 &, \text{ for } |i-j| = 2, \\
                                                       0 &, \text{ otherwise}. \\
                                                      \end{array}
 \right.
 \end{equation*}
Then, we denote by $\bfA$ the matrix with the entries $\bfA_{ij} = (\bar\bfA^{-1})_{ij} \1(|i-j|\le 2)$. One can check that the matrix $\bfA$
defined in such a way is positive semidefinite.

\item[{\bf Model 3:}] We set $\bfA_{ij} = 0$ for all the off-diagonal entries that are neither on the first row nor on the first
column of $\bfA$. The diagonal entries of $\bfA$ are
\begin{equation*}
\bfA_{11}  = p,\qquad \bfA_{ii} = 2,\quad \text{for any}\quad i \in \{2,\ldots,p\},
\end{equation*}
whereas the off-diagonal entries located either on the first row or on the first column are $\bfA_{1i} = \bfA_{i1} = \sqrt{2}$ for $i \in \{2,\ldots,p\}$.
\item[{\bf Model 4:}] We introduce the integer $k = \lceil \sqrt{p} \rceil$ and define a sparse $k\times k$ matrix
$\bar\bfA$ so that its only non-zero elements are  $\bar\bfA_{11} = k$ and, for any $i \in [2;k]$, $\bar\bfA_{ii} = 2k$ and $\bar\bfA_{1i} = \bar\bfA_{i1} = \sqrt{2}$. Then, we set
     \begin{equation*}
     \bfA = \left(  \begin{array}{cc}
										\bar\bfA & 0 \\
										0        & \bfI_{p-k}
										\end{array}
						\right).
		\end{equation*}
 \item[{\bf Model 5:}] We introduce $k = \lceil \sqrt{p} \rceil$ and define a sparse $k\times k$ matrix $\bar\bfA$ so that its only non-zero elements are $\bar\bfA_{11} = 50$ and, for any $i \in [2;k]$, $\bar\bfA_{ii} = 5$  and $\bar\bfA_{1i} = \bar\bfA_{i1} = 5/2$.
Then, similarly to previous model, we set
     \begin{equation*}
     \bfA = \left(  \begin{array}{cc}
										\bar\bfA & 0 \\
										0        & \bfI_{p-k}
										\end{array}
						\right).
		\end{equation*}
 \item[{\bf Model 6:}] We set $k = 6$, $p' = k\lceil {p}/{k} \rceil$ and define the $k\times k$ matrix $\bar\bfA$ as in model 5 above. Then,
we build the $p'\times p'$ block-diagonal matrix $\bfA$  by
 \begin{equation*}
 \bfA = \underbrace{\left( {\renewcommand{\arraystretch}{0}
                                         \setlength{\arraycolsep}{2pt}
                                        \begin{array}{ccc}
                                    \bar\bfA & & 0 \\
                                    & \ddots &  \\
                                    0 & & \bar\bfA \\
                                        \end{array}}
                                 \right)}_{\lceil {p}/{k} \rceil -\text{times}}.
\end{equation*}
Note that, in general, the resulting precision matrix in this model is not of size $p\times p$ but of size $p'\times p'$ with
$p'= 6 \lceil {p}/{6} \rceil$. However, since in the experiments reported in this section $p$ is always a multiple of $6$, we have $p=p'$.
\end{description}

In this experimental evaluation, we compare the performance of the following four estimators---introduced in previous sections---of the diagonal elements of the precision matrix:
\begin{itemize}
\item[$\bullet$] RV corresponds to the residual variance estimator defined in Section~\ref{sect:rv}.
\item[$\bullet$] RML corresponds to the relaxed maximum likelihood estimator described by equation (\ref{RML}).
\item[$\bullet$] SML corresponds to the symmetry-enforced maximum likelihood estimator described in Algorithm~\ref{algo:SPT}.
\item[$\bullet$] PML corresponds to the penalized maximum likelihood estimator described by equation (\ref{eq:phiPML}).
\end{itemize}

Note that all these algorithms need an estimator of the matrix $\bfB^*$ to produce an estimator of the diagonal
entries of the precision matrix. We conducted experiments in three different scenarios. The first scenario is when
the matrix $\bfB^*$ is estimated column-by-column by the square-root Lasso, using the penalization parameter
$\lambda = \sqrt{2 \log p}$. This value for $\lambda$ is commonly called the universal choice and
has proved to lead to optimal theoretical results and fairly good empirical results \citep{DalalyanC12,SunZhang12,DHMS13}.
The second scenario is when the
matrix $\bfB^*$ is estimated column-by-column by the ordinary least squares estimator applied to the covariates that correspond
to nonzero entries of the square-root Lasso estimator\footnote{A discussion on the strengths and weaknesses of this estimator
can be found in \citep{BelloniChernozhukov,Lederer2014}.} with the aforementioned value of $\lambda$. Finally, the third scenario is an unrealistic one; it corresponds to the case of a known matrix $\bfB^*$. This scenario is included in the experimental evaluation
in order to check the consistency between the theoretical and the empirical results as well as in order to better understand how
the error in estimating $\bfB^*$ impacts the quality of estimation of the diagonal entries of the precision matrix.

Thus, each configuration of our empirical study corresponds to choosing
\begin{itemize}
\item[$\bullet$] a model out of 6 models described above
\item[$\bullet$] a dimension $p\in\{30,60,90\}$
\item[$\bullet$] a sample size $n\in\{200,800,2000\}$
\item[$\bullet$] a method of estimating $\bfB^*$.
\end{itemize}
In each configuration, we computed the estimators RV, RML, SML and PML for 50 independent datasets. Using these $R=50$ replications,
we estimate the expected risk of estimating  $\bphi^*$, $\Ex (\| \bphi^* - \hat\bphi \|_2)$, by the average
$\frac1{R} \sum_{r=1}^R \| \bphi^* - \hat\bphi_{(r)} \|_2$.
In Tables 1-6, we report these averages along with the standard deviations of the errors measured by $\ell_2$-vector norm.
All the experiments were conducted in R, using the Mosek solver (see \cite{MOSEK}) for computing the square-root Lasso estimator
by second-order cone programming.

\begin{table}[ht]
\begin{tabular}{|*{10}{c|}}
\toprule
$p$ & \multicolumn{3}{c|}{30} & \multicolumn{3}{c|}{60} & \multicolumn{3}{c|}{90}\\
\midrule
$n$ & 200 & 800 & 2000 & 200 & 800 & 2000 & 200 & 800 & 2000 \\
\midrule
\multicolumn{10}{|c|}{\bf $\bfB^*$ estimated by square-root Lasso} \\
\midrule
RV & {\bf 0.273 } & {\bf 0.138 } & {\bf 0.084 } & {\bf 0.402} & {\bf 0.194} & {\bf 0.123 } & {\bf 0.524 } & {\bf 0.243 } & {\bf 0.150 } \\
& {\bf (.042)} & {\bf (.016)} & {\bf (.010)} & {\bf (.036)} & {\bf (.014)} & {\bf (.012)} & {\bf (.037)} & {\bf  (.017)} & {\bf (.011)} \\
\midrule
RML & 0.509 & 0.272  & 0.173 & 0.722  & 0.395 & 0.261 & 0.880  & 0.496 & 0.321 \\
&  (.062) & (.022) & (.018) & (.061) & (.026) & (.013) & (.069) &  (.028) &(.014) \\
\midrule
SML & 1.080 & 0.678  & 0.375  & 1.276 & 0.802 & 0.641  & 1.235  & 0.651  & 0.454  \\
& (.132) &  (.095) &  (.045) &  (.146) &  (.075) &  (.050) &  (.137) & (.052) &(.029) \\
\midrule
PML & 0.509 & 0.272  & 0.173 & 0.722  & 0.395  & 0.261 & 0.880  & 0.496  & 0.322  \\
 &  (.062) & (.021) & (.017) & (.061) & (.026) &  (.013) & (.069) & (.028) & (.014) \\
\midrule
\multicolumn{10}{|c|}{\bf $\bfB^*$ estimated by square-root Lasso followed by OLS} \\
\midrule
RV & {\bf 0.792 } & {\bf 0.144 } & {\bf 0.084 } & {\bf 2.251} & {\bf 1.857 } & {\bf 0.943 } & {\bf 3.261 } & {\bf 3.815 } & {\bf 3.689 } \\
 & {\bf (.192)} & {\bf(.051)} & {\bf  (.010)} & {\bf  (.203)} & {\bf (.161)} & {\bf  (.221)} & {\bf (.184)} & {\bf (.157)} & {\bf  (.120)} \\
\midrule
RML & {\bf 0.792} & {\bf 0.144} & {\bf 0.084} & {\bf 2.251} & {\bf 1.857 } & {\bf 0.943 } & {\bf 3.261 } & {\bf 3.815 } & {\bf 3.689 } \\
 & {\bf (.192)} & {\bf(.051)} & {\bf (.010)} & {\bf(.203)} & {\bf (.161)} & {\bf  (.221)} & {\bf  (.184)} & {\bf (.157)} & {\bf (.120)} \\
\midrule
SML & 1.211  & 0.610 & 0.336  & 2.515  & 1.956  & 1.095 & 3.415  & 3.832  & 3.700  \\
 & (.131) &(.106) & (.057) & (.189) &  (.143) & (.194) & (.175) &  (.152) & (.118) \\
\midrule
PML & 0.879  & 0.150  & {\bf 0.084 } & 2.366  & {\bf 1.857 } & {\bf 0.943 } & 3.342  & 3.816  & {\bf 3.689 } \\
 & (.175) &  (.051) & {\bf (.011)} & (.207) & {\bf (.160)} & {\bf  (.221)} & (.176) & (.157) & {\bf (.120)} \\
\midrule
\multicolumn{10}{|c|}{\bf $\bfB^*$ is estimated without error } \\
\midrule
RV & 0.267 & 0.138 & 0.084  & 0.380 & 0.192  & 0.122  & 0.476 & 0.237  & 0.148 \\
 & (.041) & (.016) & (.010) & (.036) & (.014) & (.011) &(.033) &(.018) & (.011) \\
\midrule
RML & 0.330  & 0.163  & 0.104 & 0.469  & 0.229 & 0.151  & 0.584  & 0.289 & 0.178 \\
 & (.046) & (.016) & (.013) & (.044) & (.023) &  (.013) & (.048) & (.020) &(.014) \\
\midrule
SML & {\bf 0.042 } & {\bf 0.019 } & {\bf 0.012} & {\bf 0.044 } & {\bf 0.021} & {\bf 0.011 } & {\bf 0.048} & {\bf 0.021} & {\bf 0.011 } \\
 & {\bf (.035)} & {\bf (.013)} & {\bf (.009)} & {\bf (.033)} & {\bf(.015)} & {\bf (.007)} & {\bf (.041)} & {\bf  (.017)} & {\bf (.010)} \\
\midrule
PML & 0.330& 0.163 & 0.104  & 0.470 & 0.229 & 0.151  & 0.584  & 0.289 & 0.178 \\
 & (.046) &  (.016) &  (.013) & (.044) & (.023) & (.012) &  (.048) & (.020) & (.014) \\
\bottomrule
\end{tabular}
\caption{Performance of the estimators of diagonal elements of the precision matrix in Model 3. The number of replications in each case
is $R=50$. More details on the experimental set-up are presented in Section~\ref{ssec:exp}.}
\end{table}

\begin{table}[ht]
\begin{tabular}{|*{10}{c|}}
\toprule
$p$ & \multicolumn{3}{c|}{30} & \multicolumn{3}{c|}{60} & \multicolumn{3}{c|}{90}\\
\midrule
$n$ & 200 & 800 & 2000 & 200 & 800 & 2000 & 200 & 800 & 2000 \\
\midrule
\multicolumn{10}{|c|}{\bf $\bfB^*$ estimated by square-root Lasso} \\
\midrule
RV & {\bf 0.372} & {\bf 0.184} & {\bf 0.113} & {\bf 0.526 } & {\bf 0.269 } & {\bf 0.161 } & {\bf 0.655 } & {\bf 0.327} & {\bf 0.206 } \\
 & {\bf (.066)} & {\bf (.036)} & {\bf(.023)} & {\bf (.066)} & {\bf (.035)} & {\bf(.024)} & {\bf (.076)} & {\bf(.046)} & {\bf (.025)} \\
\midrule
RML & 0.419  & 0.212 & 0.134  & 0.583  & 0.301  & 0.183  & 0.722  & 0.361  & 0.228  \\
 & (.067) & (.033) & (.020) & (.065) & (.033) & (.023) & (.074) & (.045) & (.024) \\
\midrule
SML & 0.468 & 0.228 & 0.144  & 0.664  & 0.334  & 0.201  & 0.843 & 0.405 & 0.252\\
& (.076) & (.033) &  (.020) &  (.079) &  (.032) &  (.024) &  (.095) & (.042) & (.024) \\
\midrule
PML & 0.450  & 0.224  & 0.142 & 0.622 & 0.326 & 0.198 & 0.763  & 0.394  & 0.247 \\
 & (.070) & (.032) &(.020) &(.069) & (.032) &(.024) &(.073) & (.042) & (.023) \\
\midrule
\multicolumn{10}{|c|}{\bf $\bfB^*$ estimated by square-root Lasso followed by OLS} \\
\midrule
RV & {\bf 0.368} & {\bf 0.182} & {\bf 0.113 } & {\bf 0.516} & {\bf 0.267} & {\bf 0.160} & {\bf 0.641} & {\bf 0.324 } & {\bf 0.205 } \\
 & {\bf (.065)} & {\bf (.036)} & {\bf(.023)} & {\bf (.064)} & {\bf (.035)} & {\bf (.024)} & {\bf (.075)} & {\bf(.046)} & {\bf(.025)} \\
\midrule
RML & {\bf 0.368} & {\bf 0.182} & {\bf 0.113} & {\bf 0.516)} & {\bf 0.267} & {\bf 0.160 } & {\bf 0.641 } & {\bf 0.324 } & {\bf 0.205 } \\
& {\bf (.065)} & {\bf  (.036)} & {\bf (.023)} & {\bf (.064)} & {\bf(.035)} & {\bf (.024)} & {\bf (.075)} & {\bf (.046)} & {\bf (.025)} \\
\midrule
SML & 0.392 & 0.191  & 0.118  & 0.558 & 0.286 & 0.173  & 0.712  & 0.351 & 0.220 \\
& (.069) &  (.037) &  (.025) & (.078) &(.033) & (.025) & (.084) & (.043) & (.025) \\
\midrule
PML & 0.383  & 0.188  & 0.116  & 0.539  & 0.280  & 0.169  & 0.680 & 0.343 & 0.215  \\
 &  (.067) & (.037) &(.024) &  (.067) &  (.033) &  (.024) &  (.077) & (.043) &  (.025) \\
\midrule
\multicolumn{10}{|c|}{\bf $\bfB^*$ is estimated without error } \\
\midrule
RV & 0.366 & 0.182  & 0.113 & 0.515  & 0.267  & 0.160  & 0.640  & 0.324  & 0.204  \\
 &  (.066) &(.036) & (.023) &(.065) & (.035) & (.024) & (.074) &  (.046) & (.025) \\
\midrule
RML & 0.374 & 0.187 & 0.116 & 0.524 & 0.271  & 0.163  & 0.649 & 0.330  & 0.208  \\
& (.065) & (.035) & (.023) &(.066) &(.035) & (.024) &  (.073) &  (.046) & (.025) \\
\midrule
SML & {\bf 0.352 } & {\bf 0.173 } & {\bf 0.108} & {\bf 0.500 } & {\bf 0.259 } & {\bf 0.156} & {\bf 0.624 } & {\bf 0.316 } & {\bf 0.199} \\
SML & {\bf(.065)} & {\bf (.039)} & {\bf (.024)} & {\bf (.066)} & {\bf (.036)} & {\bf (.025)} & {\bf  (.074)} & {\bf (.046)} & {\bf(.025)} \\
\midrule
PML & 0.353  & 0.174  & 0.109  & {\bf 0.500 } & {\bf 0.259 } & {\bf 0.156 } & 0.625  & 0.317  & 0.200 \\
 &  (.065) & (.039) &(.024) & {\bf (.066)} & {\bf  (.036)} & {\bf (.025)} &  (.074) & (.046) & (.025) \\
\bottomrule
\end{tabular}
 \caption{Performance of the estimators of diagonal elements of the precision matrix in Model 4. The number of replications in each case
is $R=50$. More details on the experimental set-up are presented in Section~\ref{ssec:exp}.}
\end{table}

In the ideal case when $\bfB^*$ is estimated without error (by itself), the empirical results reflect perfectly the theoretical
results of the previous sections. The comparison of the performance of the estimators indicates that the maximum likelihood estimators
SML and PML are preferable to the residual variance estimator. The maximum likelihood estimator considering symmetry constraints
outperforms all the other estimators. However, in practice when $\hat\bfB$ is obtained by the square-root Lasso without any
refinement, $\bphiRV$ outperforms all the other estimators in the vast majority of configurations. Some exceptions can be observed
in Models 5 and 6 (see the top part of Tables~\ref{tab:5} and \ref{tab:6},  where RV is slightly worse than the other procedures
for small sample sizes ($n=200$). It should be, however, acknowledged that the difference of the quality between the estimators in these
cases is not large enough to advocate for using RML, SML or PML. Note also that the RV estimator
satisfies the following simple inequality:
\begin{align*}
(\phiRV_j-\phi_j^*)^2 
        &= \big(\frac1n \|\bfX\hat\bfB_{\bullet,j}\|_2^2 - \phi_j^*)^2\\
        &\le \frac2{n^2} \big(\|\bfX\hat\bfB_{\bullet,j}\|_2^2 - \|\bfX\bfB_{\bullet,j}^*\|_2^2)^2 + 
        2\big(\frac1n \|\bfX\bfB_{\bullet,j}^*\|_2^2 - \phi_j^*)^2.
\end{align*}
The second term of the right-hand side is the error evaluated theoretically in the previous sections, while the first term
can be further bounded from above by $2\big(\frac1n \{\|\bfX\hat\bfB_{\bullet,j}\|_2^2 \vee \|\bfX\bfB_{\bullet,j}^*\|_2^2\}\big)
\big(\frac1n\|\bfX(\hat\bfB_{\bullet,j}-\bfB^*_{\bullet,j})\|_{2}^2\big)$. This inequality partly explains the behavior of the RV-estimator in 
the reported numerical results. More importantly, it shows that the error of estimating the matrix $\bfB^*$ might have 
a strong impact on the quality of estimating the diagonal elements. 

It is interesting to observe what happens when an additional step of estimation of $\bfB^*$ using the ordinary least squares
on the sparsity pattern provided by the square-root Lasso is performed. The impact of this step is not the same in all the
models under consideration. In particular, the quality of estimation is mostly improved for all the four estimators in models 1 and 2.
Furthermore, thanks to this variable selection step, the maximum-likelihood-type estimators perform nearly as well as the residual
variance estimator RV. In model 3, the variable selection step deteriorates the quality of estimation in most configurations, whereas
in models 4-6 this step has almost no consequence on the estimation accuracy.

The graphics of Figure \ref{fig:globfig} are drawn for Model 2 with $p = 60$.  The left plot
corresponds to the estimation error---measured by $\ell_2$-vector norm---as a function of the sample size in the scenario
$\hat\bfB=\bfB^*$, whereas the central plot corresponds to the same error when $\bfB^*$ is estimated by the OLS on the sparsity
pattern furnished by the square-root Lasso. The right plot is just a zoom on the center plot.
These plots illustrate the convergence to zero of the error of estimation for the estimators considered in this paper. The
speed of convergence in these empirical results, as expected, is nearly $n^{-1/2}$ for fixed dimension $p$.

\begin{table}[ht]
\begin{tabular}{|*{10}{c|}}
\toprule
$p$ & \multicolumn{3}{c|}{30} & \multicolumn{3}{c|}{60} & \multicolumn{3}{c|}{90}\\
\midrule
$n$ & 200 & 800 & 2000 & 200 & 800 & 2000 & 200 & 800 & 2000 \\
\midrule
\multicolumn{10}{|c|}{\bf $\bfB^*$ estimated by square-root Lasso} \\
\midrule
RV & 0.384  & {\bf 0.202} & {\bf 0.125 } & 0.543  & {\bf 0.279 } & {\bf 0.185 } & 0.701  & {\bf 0.342} & {\bf 0.222 } \\
 & (.077) & {\bf (.029)} & {\bf (.023)} &  (.060) & {\bf (.039)} & {\bf  (.024)} & (.064) & {\bf  (.040)} & {\bf (.021)} \\
\midrule
RML & {\bf 0.380} & 0.206 & 0.128 & {\bf 0.539 } & 0.287  & 0.190 & {\bf 0.697 } & 0.352  & 0.230  \\
 & {\bf (.076)} & (.027) & (.023) & {\bf (.060)} &(.040) & (.025) & {\bf (.064)} & (.041) & (.021) \\
\midrule
SML & {\bf 0.380} & 0.205 & 0.131 & {\bf 0.539 } & 0.290  & 0.194  & {\bf 0.697 } & 0.353  & 0.233  \\
 & {\bf (.076)} & (.029) &(.024) & {\bf(.060)} & (.042) & (.024) & {\bf  (.064)} & (.041) &  (.024) \\
\midrule
PML & {\bf 0.380} & 0.206 & 0.128  & {\bf 0.539} & 0.287  & 0.190  & {\bf 0.697 } & 0.352  & 0.230  \\
 & {\bf (.076)} & (.027) &(.023) & {\bf (.060)} & (.040) & (.025) & {\bf(.064)} & (.041) &  (.021) \\
\midrule
\multicolumn{10}{|c|}{\bf $\bfB^*$ estimated by square-root Lasso followed by OLS} \\
\midrule
RV & {\bf 0.379 } & {\bf 0.209 } & {\bf 0.130 } & {\bf 0.534 } & {\bf 0.295 } & {\bf 0.194} & {\bf 0.693 } & {\bf 0.367} & {\bf 0.235} \\
 & {\bf  (.076)} & {\bf (.029)} & {\bf (.025)} & {\bf (.061)} & {\bf (.040)} & {\bf (.027)} & {\bf (.064)} & {\bf(.044)} & {\bf  (.025)} \\
\midrule
RML & {\bf 0.379 } & {\bf 0.209 } & {\bf 0.130 } & {\bf 0.534 } & {\bf 0.295 } & {\bf 0.194} & {\bf 0.693 } & {\bf 0.367} & {\bf 0.235 } \\
 & {\bf (.076)} & {\bf(.029)} & {\bf(.025)} & {\bf (.061)} & {\bf (.040)} & {\bf(.027)} & {\bf (.064)} & {\bf (.044)} & {\bf (.025)} \\
\midrule
SML & {\bf 0.379} & {\bf 0.209} & 0.134 & {\bf 0.534} & 0.297 & 0.199  & {\bf 0.693 } & 0.368  & 0.241\\
 & {\bf (.076)} & {\bf (.031)} & (.026) & {\bf (.061)} & (.041) & (.027) & {\bf  (.064)} &(.043) & (.027) \\
\midrule
PML & {\bf 0.379 } & {\bf 0.209 } & {\bf 0.130 } & {\bf 0.534 } & {\bf 0.295 } & {\bf 0.194 } & {\bf 0.693 } & {\bf 0.367 } & 0.236 \\
 & {\bf  (.076)} & {\bf (.029)} & {\bf  (.025)} & {\bf  (.061)} & {\bf (.040)} & {\bf  (.027)} & {\bf (.063)} & {\bf (.043)} & (.025) \\
\midrule
\multicolumn{10}{|c|}{\bf $\bfB^*$ is estimated without error } \\
\midrule
RV & 0.384  & 0.201  & 0.125  & 0.530  & 0.275  & 0.184  & 0.686  & 0.339 & 0.221  \\
 & (.075) & (.030) &  (.022) & (.060) &  (.038) & (.023) & (.066) & (.040) & (.022) \\
\midrule
RML & 0.383  & 0.201  & 0.126 & 0.531 & 0.277 & 0.184  & 0.687  & 0.339  & 0.221  \\
 &(.076) & (.029) &(.022) & (.061) & (.037) &  (.023) &  (.066) &  (.040) & (.022) \\
\midrule
SML & {\bf 0.347 } & {\bf 0.180 } & {\bf 0.112} & {\bf 0.498 } & {\bf 0.257 } & {\bf 0.170 } & {\bf 0.647} & {\bf 0.319 } & {\bf 0.206 } \\
 & {\bf (.078)} & {\bf (.032)} & {\bf (.024)} & {\bf  (.061)} & {\bf (.042)} & {\bf (.025)} & {\bf (.067)} & {\bf  (.042)} & {\bf(.023)} \\
\midrule
PML & 0.383 & 0.201  & 0.126& 0.531 & 0.277& 0.184  & 0.687  & 0.339  & 0.221  \\
 &  (.076) &  (.029) & (.022) & (.061) & (.037) & (.023) &  (.066) & (.040) &(.022) \\
\bottomrule
\end{tabular}
 \caption{Performance of the estimators of diagonal elements of the precision matrix in Model 5. The number of replications in each case
is $R=50$. More details on the experimental set-up are presented in Section~\ref{ssec:exp}.}
\label{tab:5}
\end{table}

\begin{table}[ht]
\begin{tabular}{|*{10}{c|}}
\toprule
$p$ & \multicolumn{3}{c|}{30} & \multicolumn{3}{c|}{60} & \multicolumn{3}{c|}{90}\\
\midrule
$n$ & 200 & 800 & 2000 & 200 & 800 & 2000 & 200 & 800 & 2000 \\
\midrule
\multicolumn{10}{|c|}{\bf $\bfB^*$ estimated by square-root Lasso} \\
\midrule
RV & 0.383  & {\bf 0.207 } & {\bf 0.140 } & 0.534  & {\bf 0.310 } & {\bf 0.205 } & 0.651  & {\bf 0.374 } & {\bf 0.255 } \\
 &  (.059) & {\bf (.031)} & {\bf(.018)} &  (.054) & {\bf (.031)} & {\bf (.017)} &  (.057) & {\bf  (.034)} & {\bf (.018)} \\
\midrule
RML & {\bf 0.378 } & 0.223  & 0.157  & {\bf 0.531 } & 0.335  & 0.236  & {\bf 0.648 } & 0.408  & 0.299  \\
 & {\bf (.058)} & (.030) & (.020) & {\bf(.052)} &(.033) & (.020) & {\bf (.055)} & (.036) & (.019) \\
\midrule
SML & {\bf 0.378} & 0.229  & 0.169  & {\bf 0.531} & 0.339  & 0.249  & 0.649  & 0.410  & 0.312 \\
 & {\bf (.058)} & (.030) &(.022) & {\bf(.052)} & (.036) & (.021) & (.055) &(.036) &  (.019) \\
\midrule
PML & {\bf 0.378 } & 0.223  & 0.157  & {\bf 0.531} & 0.335  & 0.236  & {\bf 0.648 } & 0.408  & 0.299 \\
 & {\bf (.058)} & (.030) &(.020) & {\bf (.052)} & (.033) & (.020) & {\bf(.055)} & (.036) & (.019) \\
\midrule
\multicolumn{10}{|c|}{\bf $\bfB^*$ estimated by square-root Lasso followed by OLS} \\
\midrule
RV & {\bf 0.383} & {\bf 0.245} & {\bf 0.170} & {\bf 0.534} & {\bf 0.373 } & {\bf 0.262 } & {\bf 0.649} & {\bf 0.453 } & {\bf 0.341} \\
& {\bf (.058)} & {\bf (.030)} & {\bf (.019)} & {\bf(.053)} & {\bf (.030)} & {\bf  (.024)} & {\bf (.053)} & {\bf (.033)} & {\bf (.025)} \\
\midrule
RML & {\bf 0.383} & {\bf 0.245} & {\bf 0.170} & {\bf 0.534} & {\bf 0.373} & {\bf 0.262} & {\bf 0.649} & {\bf 0.453} & {\bf 0.341 } \\
 & {\bf (.058)} & {\bf (.030)} & {\bf  (.019)} & {\bf (.053)} & {\bf (.030)} & {\bf (.024)} & {\bf (.053)} & {\bf (.033)} & {\bf(.025)} \\
\midrule
SML & {\bf 0.383 } & 0.251  & 0.186  & {\bf 0.534} & 0.375 & 0.281  & {\bf 0.649} & 0.454 & 0.357  \\
 & {\bf (.058)} &  (.027) & (.022) & {\bf  (.053)} & (.030) &  (.023) & {\bf (.053)} & (.033) &  (.026) \\
\midrule
PML & 0.385 & {\bf 0.245} & {\bf 0.170} & {\bf 0.534} & {\bf 0.373} & {\bf 0.262 } & 0.650  & {\bf 0.453 } & {\bf 0.341 } \\
 &  (.057) & {\bf (.029)} & {\bf (.019)} & {\bf  (.053)} & {\bf  (.030)} & {\bf (.024)} &  (.053) & {\bf (.033)} & {\bf (.025)} \\
\midrule
\multicolumn{10}{|c|}{\bf $\bfB^*$ is estimated without error } \\
\midrule
RV & 0.408 & 0.210 & 0.141 & 0.569 & 0.309& 0.205 & 0.697 & 0.370  & 0.251  \\
 & (.068) & (.030) &  (.018) &  (.068) &  (.031) & (.018) &  (.063) &  (.030) & (.018) \\
\midrule
RML & 0.411 & 0.212 & 0.142 & 0.578 & 0.313  & 0.208  & 0.702  & 0.372  & 0.254 \\
 & (.070) &  (.030) & (.019) &(.067) & (.033) &  (.018) &  (.064) & (.031) & (.018) \\
\midrule
SML & {\bf 0.182} & {\bf 0.097} & {\bf 0.061} & {\bf 0.277} & {\bf 0.142} & {\bf 0.094} & {\bf 0.311} & {\bf 0.178 } & {\bf 0.110} \\
 & {\bf (.057)} & {\bf (.023)} & {\bf (.020)} & {\bf (.064)} & {\bf (.033)} & {\bf (.022)} & {\bf (.073)} & {\bf (.030)} & {\bf(.019)} \\
\midrule
PML & 0.411 & 0.212 & 0.142 & 0.578 & 0.313 & 0.208 & 0.702  & 0.372  & 0.254 \\
 &  (.070) &  (.030) &  (.019) & (.067) &  (.033) &  (.018) &  (.064) & (.031) & (.018) \\
\bottomrule
\end{tabular}
 \caption{Performance of the estimators of diagonal elements of the precision matrix in Model 6. The number of replications in each case
is $R=50$. More details on the experimental set-up are presented in Section~\ref{ssec:exp}.}
\label{tab:6}
\end{table}

\subsection{Details on the implementation}\label{ssec:impl}

\paragraph{Symmetry-enforced maximum likelihood.}

As we explained earlier, the product structure of the term $\bfDelta^*_j$ in (\ref{delta}) may cause the amplification of the
estimation error when passing from $\hat\bfB$ to $\hat\bphi$. In order to reduce as much as possible this phenomenon, we suggested
to choose the path $\mathscr C$ by minimizing its length. In addition, the fact that some entries of $\bfB^*$ appear in the
denominator of $\bfDelta^*_j$, make it unsuitable to include in $\mathscr C$ edges corresponding to small values of
$\hat\bfB_{ij}$. The combination of these two arguments suggests to define edge weights as decreasing functions of
$\hat\bfB_{ij}$ and to look for paths that somehow minimize the overall weight defined as the sum of the weights of the
edges contained in $\mathscr C$.

The two versions of the SML algorithm that have been implemented and tested in this work make use of the minimum spanning tree (MST)
and the shortest path tree in the step of determining the way of computation the elements of $\hat\bphi$ belonging to a connected
component $\mathscr C$ of the graph $\hat{\mathscr G}$. A MST of $\mathscr C$ is a tree that spans $\mathscr C$
and has the smallest total weight among all the spanning trees of $\mathscr C$. The shortest path tree having a given node $r$ as a root
is a spanning tree $\mathscr T$ of $\mathscr C$ such that for any node $j\in\mathscr C$ the weight of the path from $j$ to $r$ in
$\mathscr T$ is the smallest among the weights of all possible paths from $j$ to $r$ in $\mathscr C$.

We have used the Kruskal \citep{Kruskal} algorithm  for finding the MST and the Jarnik-Prim-Dijkstra  algorithm
\citep{jarnik,Prim,Dijkstra} for the shortest path tree.  The worst-case computational complexities of the construction of these
trees are the following \citep{Cormen}. When the graph $\mathscr G$ has $p$ nodes and $q$ edges, the Kruskal algorithm runs in $O(q\log p)$ time.
Its output is a set of MSTs per connected component. The version of the SML based on the shortest path tree requires $O(p + q)$
operations to find the connected components. In a connected component having $p_c$ nodes and $q_c$ edges, the node of largest
degree can be obtained in $O(q_c)$ operations, while the computational complexity of finding the shortest paths from a node to
all the others is $O(q_c \log(p_c))$. Therefore, determining a shortest path tree per connected component has a complexity of
$O(p + q\log(p))$, or $O(sp\log(p))$ where $s$ is the maximal degree of a node of $\hat{\mathscr G}$. Thus, the computational
complexities of the two versions of the SML estimator are comparable and, at most, of the order $O(sp\log(p))$.

In our experiments, we have also tried\footnote{We used the package RBGL of R \citep{RBGL} for various algorithms related to weighted
graphs.} a third version consisting in computing the shortest path trees from every node of
a connected component and then choosing the one with the minimal overall weight, rather than first choosing the root as the
node having largest degree.  Several other variants have been tested as well, but the simplest version based on choosing the
MST has lead to the best empirical results.

\begin{figure}
\vspace{-25pt}
\input{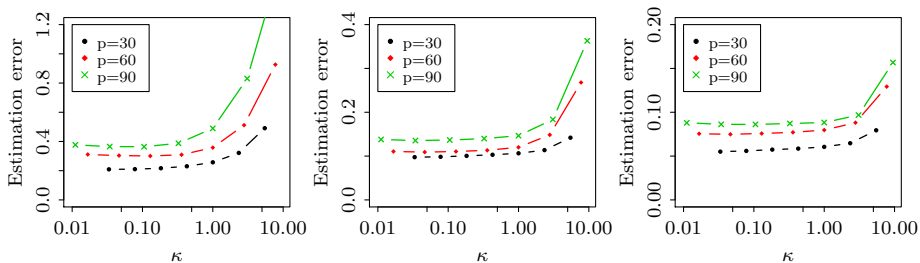}
\vspace{-24pt}
\caption{The estimation error of the PML as a function of $\kappa$. The plots are obtained for the synthetic experiment
of Model 2 with various values of $p$ and for $n=200$ (left), $n=800$ (middle) and $n=2000$. Please note that the limits of
the $y$-axis are not the same in the three plots and that the $x$-axis is presented in logarithmic scale.}
\label{fig:2}
\end{figure}

\paragraph{Penalized maximum likelihood.}

As mentioned earlier, the PML estimator is computed by solving the optimization problem \eqref{eq:f}.
We implement a steepest descent algorithm with adaptive step-size and scaled descent direction $- \nabla f(\bv_h) / \| \nabla f(\bv_h) \|_2$.
At each iteration, one common adaptation for every coordinate of the descent direction is performed.
If the objective function increases, the current iteration is done again with a halved step-size.
On the opposite, if the objective function decreases, the step-size is increased by a constant factor
for the next iteration.

Mathematically speaking, the update operations for our gradient descent algorithm are
\begin{equation}
 \bv_0 = \bfone , \quad \bv_{h+1} = \bv_h + t_h \bu_h ,\qquad h=0,1,2,\ldots,
\end{equation}
where the descent direction is $\bu_h = - \nabla f(\bv_h)/ \| \nabla f(\bv_h) \|_2$ and $t_h$ is the step-size.
Thanks to the convexity, the convergence of this algorithm is guaranteed for any starting point $\bv_0$.
The step-size is updated at each iteration according to the following rule:
\begin{align*}
t_{h+1} = \left\lbrace \begin{array}{ll}
                  1.2 \times t_h , & \qquad \text{for } f(\bv_{h+1}) < f(\bv_h) , \\
                  0.5 \times t_h, &\qquad \text{otherwise.}
                 \end{array}
	  \right.
\end{align*}
The multiplicative factors we use for adaptive step-size are those propose by \cite{Riedmiller92rprop} for the Rprop algorithm.
We stop iterating when the gradient magnitude measured in the $\ell_2$-norm is below a certain level ($10^{-5}$ in our experiments)
or when the limit of 5000 iterations is attained.

For the choice of the tuning parameter $\kappa$, we did a cross-validation by choosing a geometric grid over the values of $\kappa$
ranging from $1/p$ to $\sqrt{p}$. The results, for Models 2 and 4, are plotted in Fig.~\ref{fig:2} and \ref{fig:3}, respectively.
We can clearly see that there is a large interval of values of $\kappa$ for which the error is nearly minimal. Based on this observation,
we chose $\kappa = \frac13\sqrt{\log p}$ for all the numerical experiments reported in Tables~\ref{tab:1}-\ref{tab:6}.

\section{Conclusion}

This paper introduces three estimators of the diagonal entries of a sparse precision matrix when $n$ iid copies of a Gaussian vector
with this precision matrix are observed. The properties of these estimators are discussed and compared with those of the commonly used
residual variance estimator. At a theoretical level, an interesting finding is that the naive maximum likelihood estimator (MLE) that does
not take into account the symmetry constraints has a significantly larger risk than the residual variance estimator and, hence, is not
optimal even asymptotically. The symmetry-enforced MLE and the penalized MLE circumvent this drawback and are shown in all numerical 
experiments to outperform the residual variance estimator when the matrix $\bfB^*$ is known. Similar but unreported results are obtained
when the estimators of the diagonal entries use a noisy matrix $\hat\bfB= \bfB^* + \bfXi$, provided the noise matrix $\bfXi$ has iid
Gaussian entries with zero mean and small variance. However, in a more realistic situation when $\bfB^*$ is estimated by the square-root
Lasso or by the ordinary least squares conducted over the submodel selected by the square-root Lasso, the accuracies of the four estimators 
of the diagonal entries become comparable with a slight advantage for the residual variance estimator. 

We would like also to mention the introduction of a novel and simple method of estimating partial correlations and of symmetrizing the precision matrix estimator derived from the nonsymmetric matrix $\hat\bfB$. It is based on the observation that the square of the partial correlation between $i$-th and $j$-th variables is equal to $\bfB^*_{ij}\bfB^*_{ji}$.

In the future, it would be interesting to look for an estimator of $\bfB^*$ which is more accurate than the square-root Lasso and could
hopefully---in combination with the symmetry-enforced MLE or the penalized MLE---lead to better precision matrix estimate than the
one obtained by the association of the square-root Lasso and the residual variance estimator. Another appealing avenue for future research
is the investigation of the case when the matrix $\bfX$ is observed with an error. Recent papers \citep{Rosenbaum1,BelloniRT}
may provide valuable guidance for accomplishing this task.

\begin{figure}
\vspace{-25pt}
\input{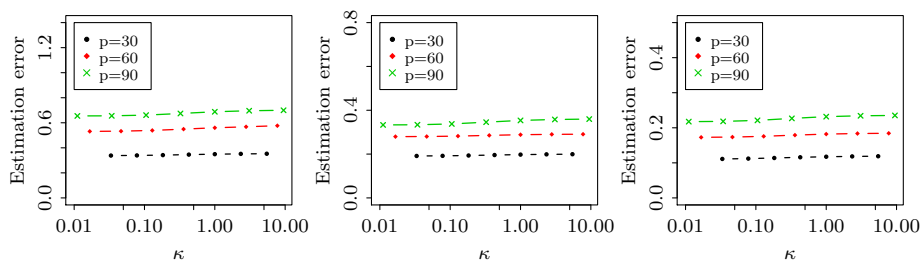}
\vspace{-24pt}
\caption{The estimation error of the PML as a function of $\kappa$. The plots are obtained for the synthetic experiment
of Model 4 with various values of $p$ and for $n=200$ (left), $n=800$ (middle) and $n=2000$. Please note that the limits of
the $y$-axis are not the same in the three plots and that the $x$-axis is presented in logarithmic scale.}
\label{fig:3}
\end{figure}

\section*{Acknowledgments} The work of the second author was partially supported by the grant Investissements d'Avenir (ANR-
11-IDEX-0003/Labex Ecodec/ANR-11-LABX-0047) and the chair ``LCL/GENES/Fondation du risque, Nouveaux enjeux pour nouvelles donn\'ees''.

\bibliographystyle{plainnat}
\bibliography{diagElementsBibliography}

\end{document}